\documentclass[a4paper,11pt,oneside]{article}
\usepackage[utf8]{inputenc}
\usepackage[T1]{fontenc}
\usepackage[english]{babel}
\usepackage{amsthm}
\usepackage{array}
\usepackage{tikz}
\usepackage{appendix}
\usepackage{amsmath}
\usepackage{amssymb}
\usepackage{hyperref}
\usepackage{mathrsfs} 
\usepackage{bbm}
\newcommand{\xdownarrow}[1]{
  {\left\downarrow\vbox to #1{}\right.\kern-\nulldelimiterspace}
}
\hypersetup{
    colorlinks=true, 
    linktoc=all,
    urlcolor=darkgray,   
    linkcolor=blue,  
    citecolor=darkgray,
}
\usepackage{babel}
\usepackage{csquotes}
\definecolor{greenpigment}{rgb}{0.0, 0.65, 0.31}
\definecolor{iceberg}{rgb}{0.44, 0.65, 0.82}

\usepackage{graphicx, sidecap}
\usepackage{algorithm}
\usepackage{algpseudocode}
\usepackage{wasysym}
\usepackage{pifont,stmaryrd}

\usepackage{mathtools}
\usepackage{amsfonts}
\usepackage{multicol}
\usepackage{setspace}
\usepackage{listings}
\usepackage{tikz}
\usetikzlibrary{arrows}
\usetikzlibrary{positioning}
\usepackage{caption}
\usetikzlibrary{shapes,arrows}
\usepackage{multirow}
\usepackage{setspace}
\usepackage[top=1.5cm, bottom=3cm, left=2cm , right=1cm]{geometry}
\usepackage{fancybox}
\usepackage{xcolor, color}
\usepackage{pifont}
\reversemarginpar
\usepackage{url}

\usepackage[all]{xy}
\usepackage{tikz-cd}

\usepackage{amssymb}

\newcommand{\R}{\mathbb{R}}
\newcommand{\C}{\mathbb{C}}
\newcommand{\N}{\mathbb{N}}

\newcommand{\E}{\mathbb{E}}

\newcommand{\U}{\mathcal{U}}
\newcommand{\DD}{\mathcal{D}}
\newcommand{\g}{\mathfrak{g}}
\newcommand{\kk}{\mathfrak{k}}
\newcommand{\aL}{\mathfrak{a}}
\newcommand{\End}{\text{End}}
\newcommand{\Hom}{\text{Hom}}
\newcommand{\Hol}{\text{Hol}}
\newcommand{\Pol}{\text{Pol}}

\newcommand{\Ad}{\text{Ad}}

\newcommand{\Repart}{\text{Re}}
\newcommand{\supp}{\text{supp}}

\newcommand{\Tr}{\text{Tr}}
\newcommand{\Id}{\text{Id}}

\newcommand*{\QEDA}{\null\nobreak\hfill\ensuremath{\square}}

\usepackage{cancel}
\theoremstyle{plain}

\numberwithin{equation}{section}

\theoremstyle{plain}
\newtheorem{cor}{Corollary}

\newtheorem{prop}{Proposition}
\newtheorem{thm}{Theorem}

\newtheorem{lem}{Lemma}
\newtheorem{defn}{Definition}
\theoremstyle{definition}
\newtheorem{exmp}{Example}

\theoremstyle{remark}
\newtheorem{req}{Remark}

\title{\textbf{A topological Paley-Wiener-Schwartz Theorem for sections of homogeneous vector bundles on $G/K$}}
\author{Martin \textsc{Olbrich} and Guendalina \textsc{Palmirotta}}
\date{} %{\today{}}

\begin{document}

\maketitle

%Abstract
\abstract
We study the Fourier transform for compactly supported distributional sections of complex homogeneous vector bundles on symmetric spaces of non-compact type $X=G/K$. 
We prove a characterisation of their range. 
In fact, from Delorme's Paley-Wiener theorem for compactly supported smooth functions on a real reductive group of Harish-Chandra class, we deduce topological Paley-Wiener and Paley-Wiener-Schwartz theorems for sections.
\tableofcontents

%Section 0 - Introduction
\section{Introduction}

%FT and PW - Euclidean case
One of the central theorems of harmonic analysis on $\R^n$, is the so-called Paley-Wiener theorem, named after the two mathematicians Raymond Paley and Norbert Wiener.
It describes the image of the Fourier transform
of the space $C^\infty_c(\R^n)$ of smooth functions with compact support as the space of entire functions on $\C^n$ satisfying some growth condition. 
The theorem has a counterpart, known as Paley-Wiener-Schwartz theorem. Here, the smooth functions are replaced by distributions $T \in C^{-\infty}_c(\R^n)$ and the growth condition by a weaker growth condition (e.g. \cite{Hormander1}, Thm. 7.3.1). 

Both theorems have been generalized to more general Lie groups $G$ and furthermore to some smooth manifolds carrying symmetries.
For example, the case of Riemannian symmetric spaces of non-compact type $X=G/K$ was considered by Helgason \cite{HelgasonA1} and Gangolli \cite{Gangolli}. 
They proved a Paley-Wiener theorem for compactly supported $K$-invariant smooth functions and Helgason \cite{HelgasonA2} even showed it for general compactly supported smooth functions on $X$. 
There is also a Paley-Wiener theorem for $K\times K$-finite compactly supported smooth functions on a real reductive Lie group $G$ of Harish-Chandra class due to Arthur \cite{Arthur} and Delorme \cite{Delorme}, formulated in terms of the so-called Arthur-Campolli and Delorme conditions, respectively. 
Delorme even proved a version without the $K \times K$-finiteness.
A generalization to $K$-finite functions on reductive symmetric spaces was presented by van den Ban and Schlichtkrull \cite{vandenBanDist2}.
Furthermore, later van den Ban and Souaifi \cite{vandenBan} proved, without using the proof or validity of any associated Paley-Wiener theorems of Arthur or Delorme, that the two compatibility conditions are equivalent.
Concerning the Paley-Wiener-Schwartz theorem for distributions on symmetric spaces, we mention Helgason \cite{HelgasonA2} and Eguchi, Hashizume, Okamato \cite{Eguchi}.
Moreover, van den Ban and Schlichtkrull \cite{vandenBanDist} also proved a topological Paley-Wiener-Schwartz theorem for $K$-finite distributions on reductive symmetric spaces.\\
%Goal and contents
Our aim is to establish a topological Paley-Wiener theorem for (distributional) sections of homogeneous vector bundles on $X$ using Delorme's intertwining conditions.
Thus, starting, in Section~\ref{sect:DelormePW} with Delorme's Paley-Wiener theorem (\cite{Delorme}, Thm.~2) in the setting  of van den Ban and Souaifi \cite{vandenBan}, we will adjust it, in Sections~\ref{sect:FT} and \ref{sect:DelormeIntw}, for our proposes.
More precisely, we describe the intertwining conditions for sections and show that there are equivalent with Delorme's one by using Frobenius-reciprocity (Prop.~\ref{prop:Level01} \& Thm.~\ref{thm:equivintcond}).
We consider three levels, (\textcolor{blue}{Level~1}) refers to Delorme's Paley-Wiener theorem (Thm.~\ref{thm:Delorme1}), (\textcolor{blue}{Level~2}) corresponds to the desired Paley-Wiener theorem for sections (Thm.~\ref{thm:PWsect}) and (\textcolor{blue}{Level~3}) stands for the Paley-Wiener theorem for 'spherical functions' (Thm.~\ref{thm:PWsect}).
For the last, we fixed an irreducible $K$-representations on the left while a right, not necessary irreducible, $K$-type $*$ is fixed by the bundle $\E_* \rightarrow X$.
In this way, it will be much easier to manage the intertwining conditions. \\
Finally in Section~\ref{sect:PWS}, we present, a topological Paley-Wiener-Schwartz theorem for distributional sections (Thm.~\ref{thm:PWSsect}) in both levels (\textcolor{blue}{Level 2}) and (\textcolor{blue}{Level 3}).
We used van den Ban and Schlichtkrull's technique \cite{vandenBanDist} as well as Camporesi's Plancherel theorem for sections (\cite{Camporesi}, Thm.~3.4 \& Thm.~4.3).\\
This paper ends, in Section~\ref{sect:ImpactDO}, by analysing consequences of this theorem for linear invariant differential operators between sections of homogeneous vector bundles (Prop.~\ref{prop:1}).

The motivation behind this work lies in solvability questions of systems of invariant differential equations on symmetric spaces $G/K$. In fact, the results of the present paper as well as applications to solvability questions are part of the doctoral dissertation \cite{Palmirotta} of the second author. For further details, we refer to \cite{Palmirotta} and the upcoming papers (\cite{PalmirottaIntw}, \cite{PalmirottaSolv}).

%Section 1
\section{On Delorme's Paley-Wiener Theorem} \label{sect:DelormePW}
%Notations
Let $G$ be a real connected semi-simple Lie group with finite center of non-compact type with Lie algebra $\g$ and $K \subset G$ its maximal compact subgroup with Lie algebra $\kk$.
The quotient $X=G/K$, then is a Riemannian symmetric space of non-compact type.

Let $\g=\kk \oplus \mathfrak{p}$ be the Cartan decomposition, and let $\aL$ be a maximal abelian subspace of $\mathfrak{p}$.
Fix a corresponding minimal parabolic subgroup $P=MAN$ of $G$ with split component $A=\exp(\aL)$, nilpotent Lie group $N$ and $M=Z_K(\aL)$ being the centralizer of $A$ in $K$.
Let $(\sigma,E_\sigma) \in \widehat{M}$ be a finite-dimensional irreducible representation of $M$ and $\lambda \in \aL^*_\C \cong \C^n$. 
For fixed $(\sigma,\lambda) \in \widehat{M} \times \aL^*_\C$, let $(\sigma_{\lambda}, E_{\sigma,\lambda})$ be the representation of $P$ on the vector space $E_{\sigma,\lambda}=E_\sigma$
such that
$\sigma_\lambda (man)=a^{\lambda+\rho}\sigma (m) \in \End(E_{\sigma,\lambda})$
for $m\in M$, $a\in A, n\in N$ and where $\rho \in \aL^+$ is the half sum of the positive roots of $(\g,\aL)$, counted with multiplicities.
 We use the notation $a^\lambda$ for $e^{\lambda \log(a)}$.
Then, the space
$$H^{\sigma,\lambda}_{\infty}:=\{f: G \stackrel{C^\infty}{\rightarrow} E_{\sigma,\lambda} \;|\; f(gman)=a^{-(\lambda+\rho)}\sigma(m)^{-1}(f(g))\} \cong C^\infty(G/P, \E_{\sigma,\lambda})$$
together with the left regular action
$(\pi_{\sigma,\lambda}(g) f)(x):=f(g^{-1}x)=(l_gf)(x)$ for $g,x \in G$ and $f \in H^{\sigma,\lambda}_\infty$,
is the space of smooth vectors of the principal series representations of $G$ induced from the $P$-representation $\sigma_\lambda$ on $E_{\sigma,\lambda}$ (e.g. \cite{Knapp2}, p. 168).
The restriction map from $H^{\sigma,\lambda}_{\infty}$ to functions on $K$ is injective by the Iwasawa decomposition $g=\kappa(g)e^{a(g)}n(g) \in KAN$ of $G$.
In particular, for $f \in H^{\sigma,\lambda}_{\infty}$ we have
$ f(g)=f(\kappa(g)e^{a(g)}n(g)) = a(g)^{-(\lambda+\rho)}(f(\kappa(g))).$
This yields, the so-called \textit{compact picture} of $H^{\sigma,\lambda}_{\infty}$ (e.g. \cite{Knapp2}, p.~168).
 It has the advantage that the representation space 
\begin{equation} \label{eq:Hsigmainfty}
H^{\sigma}_\infty := \{\varphi: K \stackrel{C^\infty}{\rightarrow} E_{\sigma} \;|\; \varphi(km)=\sigma(m)^{-1}\varphi(k), \; k \in K, m\in M\} \cong C^\infty(K/M,\E_{\sigma})
\end{equation}
does not depend on $\lambda$.
Here, $H^\sigma_\infty$ is equipped with the usual Fréchet topology. From time to time, we need the $L^2$-norm.
In the compact picture, the action of all elements $g\in G$, which are not in $K$, is slightly more involved, since  we need to commute them with the argument $k \in K$, i.e.
\begin{equation} \label{eq:repK}
(\pi_{\sigma,\lambda}(g)\varphi)(k)= a(g^{-1}k)^{-(\lambda+\rho)} \varphi(\kappa(g^{-1}k)),\;\;\;\;\varphi \in H^\sigma_\infty.
\end{equation}

%FT Level 1
\subsubsection*{Fourier transform for $G$ in (\textcolor{blue}{Level 1})}
Let
\begin{equation*} \label{eq:Cclosedball}
C^\infty_c(G)=\bigcup_{r> 0} C^\infty_r(G):=\bigcup_{r> 0} \{f \in C^\infty(G)\;|\; supp(f) \in \overline{B}_r(o)\}
\end{equation*}
be the space of compactly supported smooth complex functions on $G$, where 
$$\overline{B}_r(o):=\{g \in G \; | \; {dist}_X (gK,o) \leq r \} \subset G$$
denotes the preimage of the closed ball of radius $r$ and center $o=eK$ in $X$ under the projection $G \rightarrow X$.
Here, ${dist}_X$ means a fixed $G$-invariant Riemannian distance on $X$ and $e$ is the neutral element of $G$.
We equip $C^\infty_r(G)$ with the usual Fréchet topology, thus
$C^\infty_c(G)$ is a LF-space.
Given $\sigma \in \widehat{M}$, let us consider the map
$$\pi_{\sigma, \cdot}: G \rightarrow ( \aL^*_\C \rightarrow \End(H^\sigma_\infty)), g \mapsto (\lambda \mapsto \pi_{\sigma,\lambda}(g)).$$

\begin{defn}[Fourier transform for $G$ in (\textcolor{blue}{Level 1})]  \label{def:FTDelorme}
Fix $(\sigma,\lambda) \in \widehat{M} \times \aL^*_\C$, we define the Fourier transform of $f\in~C^\infty_c(G)$ by the operator
$$\mathcal{F}_{\sigma,\lambda}(f):=\pi_{\sigma,\lambda}(f)=\int_G f(g) \pi_{\sigma,\lambda}(g) \;dg \in \emph\End(H^{\sigma}_\infty).$$
\end{defn}

We denote by $\Hol(\aL^*_\C)$ the space of holomorphic functions in $\aL^*_\C$ and by $\Hol(\aL^*_\C, \End(H^\sigma_\infty))$ the space of maps
$ \aL^*_\C \ni \lambda \mapsto \phi(\lambda) \in \text{End}(H^{\sigma}_\infty)$
such that
\begin{itemize}
\item[$(1.i)$] for $\varphi\in H^{\sigma}_\infty,$ the function
$\lambda \mapsto \phi(\lambda)\varphi\in H^{\sigma}_\infty$
 is holomorphic.
\end{itemize}
From (\cite{Delorme}, Lem. 10 (ii)), we deduce the following statement.

\begin{prop}
The family of applications $f \mapsto \mathcal{F}_{\sigma,\lambda}(f)$
 is a linear map from
 $C^\infty_c(G)$ into\\$\prod_{\sigma\in \widehat{M}} \emph\Hol(\aL^*_\C, \emph\End(H^\sigma_\infty)). $\QEDA
\end{prop}

\subsubsection*{Delorme's Paley-Wiener theorem and intertwining conditions in (\textcolor{blue}{Level 1})}

We now proceed with the definition of Delorme's Paley-Wiener space (\cite{Delorme}, Def. 3).
It induced Delorme's intertwining conditions for \textit{derived} versions of $H^{\sigma}_\infty$ (\cite{Delorme}, Sect. 1.5 \& Déf.~3~(4.4)). 
Van den Ban and Souaifi present a more elegant reformulation of them (\cite{vandenBan}, Sect. 4.5, in particular Lem. 4.4. and Prop. 4.5.). In the same spirit, we present a very similar definition of derived $G$-representations.

\begin{defn}[$m$-th derived representation] \label{def:succderv}
For $\lambda \in \aL^*_\C$, let $\emph\Hol_\lambda$ be the set of germs at $\lambda$ of $\C$-valued holomorphic functions $\mu \mapsto f_\mu$
and $m_\lambda \subset \emph \Hol_\lambda$ the maximal ideal of germs vanishing at $\lambda$.\\
Denote by $H^{\sigma}_{[\lambda]}$ the set of germs at $\lambda$ of $H^{\sigma}_\infty$-valued holomorphic functions 
$\mu \mapsto \phi_{\mu} \in H^{\sigma}_\infty$ 
with $G$-action
\begin{equation*} \label{eq:Gactionm}
(g\phi)_{\mu} = \pi_{\sigma,\mu}(g) \phi_{\mu},\;\;\; g \in G.
\end{equation*}
For $m\in \N_0$, it induces a representation $\pi^{(m)}_{\sigma,\lambda}$ on the space
\begin{equation} \label{eq:succderv}
H^{\sigma,\lambda}_{\infty,(m)} := H^{\sigma}_{[\lambda]} / m_\lambda^{m+1} H^{\sigma}_{[\lambda]},
\end{equation}
which is equipped with the natural Fréchet topology.
We call this representation the $m$-th derived principal series representation of $G$.
\end{defn}

Here, $\Hol_\lambda$ acts on $H^{\sigma}_{[\lambda]}$ by pointwise multiplication.
Note that the $m=$0-th derived representation $H^{\sigma,\lambda}_{\infty,(0)} \cong H^{\sigma}_\infty$ is the space of smooth vectors of the principal series $G$-representation in the compact picture.
Intuitively, we can say that $H^{\sigma,\lambda}_{\infty,(m)}$ contains all Taylor polynomials of order $m$ at $\lambda$ of holomorphic families $\phi_{\mu}$.
Moreover, $\phi \in \prod_{\sigma\in \widehat{M}} \Hol(\aL^*_\C, \End(H^\sigma_\infty))$ induces an operator on each $H^{\sigma,\lambda}_{\infty,(m)} $.\\
The following definition turns out to be equivalent to Delorme's intertwining condition (\cite{Delorme}, Déf. 3 (4.4)).

\begin{defn}[Delorme's intertwining condition in (\textcolor{blue}{Level 1})]
\label{defn:intwcond}
Let $\Xi$
be the set of all 3-tuples $(\sigma,\lambda,m)$ with $\sigma \in \widehat{M}$, $\lambda \in  \aL^*_\C$ and $m\in\N_0$.
Consider the $m$-th derived $G$-representation $H^{\sigma,\lambda}_{\infty,(m)}$ defined in (\ref{eq:succderv}). 
For every finite sequence $\xi=(\xi_1,\xi_2,\dots,\xi_s) \in \Xi^s, s\in \N$, we define the $G$-representation
$$H_\xi:= \bigoplus_{i=1}^s H^{\sigma_i,\lambda_i}_{\infty,(m_i)}.$$
We consider proper closed $G$-subrepresentations $W \subseteq H_\xi$.\\
Such a pair $(\xi,W)$ with $\xi \in \Xi^s$ and $W \subset H_\xi$ as above, is called an intertwining datum.
Every  function $\phi \in \prod_{\sigma\in \widehat{M}} \emph \Hol(\aL^*_\C, \emph \End(H^\sigma_\infty))$ induces an element
$$\phi_\xi \in \bigoplus_{i=1}^s \emph\End(H^{\sigma_i,\lambda_i}_{\infty,(m_i)}) \subset \emph\End(H_\xi).$$
\begin{itemize}
\item[(D.a)] We say that $\phi$ satisfies Delorme's intertwining condition, if $\phi_\xi(W) \subseteq W$ for every intertwining datum $(\xi,W)$.
\end{itemize}
\end{defn}

\noindent
Next, we define Delorme's Paley-Wiener space (\cite{Delorme}, Déf. 3).
We denote by $\U(\kk)$ the universal enveloping algebra of complexification of $\kk$ (e.g. \cite{Jacobson} Chap. V).
Note that  our fixed Riemannian metric corresponds to a $Ad$-invariant bilinear form on $\g$, which is definit on $\kk$ and $\mathfrak{p}$.
Therefore, we get a norm $|\cdot|$ on $\mathfrak{b}^*_\C$ for each subspace $\mathfrak{b} \subset \kk$ or $\mathfrak{b} \subset \mathfrak{p}$.
\begin{defn}[Paley-Wiener space in (\textcolor{blue}{Level 1})] \label{def:PWL1}
For $r> 0$, Delorme's Paley-Wiener space is the vector space
\begin{equation} \label{eq:PWDelorme} 
PW_r(G):= \Big\{\phi \in \prod_{\sigma\in \widehat{M}} \emph \Hol(\aL^*_\C,  \emph\End(H^\sigma_\infty)) \;|\; \phi \text{ satisfies the growth condition } (1.ii)_r \text{ below and } (D.a)\Big\}.
\end{equation}
Here,
\begin{itemize}
\item[(1.ii)$_r$] for all $Y_1,Y_2 \in \U(\mathfrak{k}), (\sigma, \lambda) \in \widehat{M}\times \aL^*_\C$ and $N \in \N_0$, there exists a constant $C_{r,N,Y_1,Y_2}>0$ such that
$$||\pi_{\sigma,\lambda}(Y_1)\phi(\sigma,\lambda) \pi_{\sigma,\lambda}(Y_2)|| \leq C_{r,N,Y_1,Y_2}
(1+|\Lambda_{\sigma}|^2+|\lambda|^2)^{-N} e^{r|\emph \Repart(\lambda)|}$$
for $\phi \in  \emph \End(H^\sigma_\infty)$ and
where $\Lambda_{\sigma}$ is the highest weight of $\sigma$, $||\cdot||$ is the operator norms on $H^{\sigma}_\infty$ with respect to the $L^2$-norm of $H^\sigma_\infty$.
\end{itemize}
\end{defn}

Notice that, due to Lem. 10 (i) in \cite{Delorme}, the space $PW_r(G)$ equipped with semi-norms:
$$||\phi||_{r,N,Y_1,Y_2}:= \sup_{(\sigma,\lambda) \in \widehat{M} \times \aL^*_\C} 
(1+|\Lambda_{\sigma}|^2+|\lambda|^2)^{N}e^{-r|\text{Re}(\lambda)|} ||\pi_{\sigma,\lambda}(Y_1)\phi(\sigma,\lambda) \pi_{\sigma,\lambda}(Y_2)||_{H^{\sigma}_\infty}, \;\; \phi \in PW_r(G)$$
is a Fréchet space.
Furthermore, the intertwining condition $(D.a)$ in Def.~\ref{def:PWL1} is a special case of van den Ban and Souaifi's one (\cite{vandenBan}, Cor. 4.7 and Prop. 4.10.).
The small difference is, that instead of the defined $m$-th derived representations $H^{\sigma,\lambda}_{\infty,(m)}$ (\ref{eq:succderv}), they consider 
$$H^\sigma_{[\lambda],E}:= H^\sigma_{[\lambda]} \otimes_{\Hol_\lambda} E,$$
where $E$ is a finite-dimensional $\Hol_\lambda$-module. 
By the following proposition, this leads to equivalent intertwining conditions.

\begin{prop} 
With the previous notations, let $(\sigma,\lambda) \in \widehat{M} \times \aL^*_\C$.
Then, for $E=\emph \Hol_\lambda/m_\lambda^{m+1}$, we have that $H^\sigma_{[\lambda],E} \cong H^{\sigma,\lambda}_{\infty,(m)}$. \\
Moreover, for any finite-dimensional $\emph\Hol_\lambda$-module $E$,
there exists $m_1, \dots, m_s \in \N_0$ such that $H^\sigma_{[\lambda],E}$ is a quotient of $H^{\sigma,\lambda}_{\infty,(m_1)} \oplus \dots \oplus H^{\sigma,\lambda}_{\infty,(m_s)}$.
\end{prop}

\begin{proof}
Consider a (commutative) ring $R$ with neutral element 1, a $R$-module $M$ and $I \subset R$ an ideal. Then, we have the following isomorphism
$$M \otimes_R R/I \cong M/IM.$$
In fact, by an algebraic computation, one can easily show that the two maps
\begin{eqnarray*}
\alpha:M \otimes_R R/I \rightarrow M/IM &\text{ and }& \beta: M/IM \rightarrow M \otimes_R R/I\\
\alpha(m\otimes [r]):=[rm] && \beta([m]):=m \otimes [1]
\end{eqnarray*}
are well-defined and inverse to each other.
Here $[\cdot]$ denotes the class in the corresponding quotient.
For $m\in \N_0$ and $R=\Hol_\lambda$, consider its maximal ideal
$m_\lambda^{m+1} \subset \Hol_\lambda$.
Take $E= \Hol_\lambda/m_\lambda^{m+1}=R/I$ and $M=H^\sigma_{[\lambda]}$, then
$$ H^\sigma_{[\lambda]} \otimes_{\Hol_\lambda} E \cong H^\sigma_{[\lambda]}/m_\lambda^{m+1} H^\sigma_{[\lambda]}=: H_{\infty, (m)}^{\sigma,\lambda}.$$
Moreover, by their Lem. 2.1 in \cite{vandenBan}, an ideal $\mathcal{I}$ in $\Hol_\lambda$ is cofinite, if and only, if there exists $m\in \N_0$ such that $m_\lambda^{m+1} \subset \mathcal{I}$.\\
Thus, for some $s\in \N$ and finitely many cofinite ideals $m_\lambda^{m_1+1}, \dots, m_\lambda^{m_s+1}$ of $\Hol_\lambda$, we have that $E$ is a quotient of the direct sum 
$$\Hol_\lambda/m_\lambda^{m_1+1} \oplus \Hol_\lambda/m_\lambda^{m_2+1} \oplus \dots \oplus \Hol_\lambda/m_\lambda^{m_s+1}.$$
Hence, the map
$$H^{\sigma,\lambda}_{\infty,(m_1)} \oplus \dots \oplus H^{\sigma,\lambda}_{\infty,(m_s)} \longrightarrow E$$
is surjective and the result follows.
\end{proof}

Now, we can formulate Delorme's Paley-Wiener theorem.

\begin{thm}[Paley-Wiener Theorem, \cite{Delorme}, Thm.~2] \label{thm:Delorme1}
For $r> 0$, the Fourier transform
$$ C_r^\infty(G) \ni f \mapsto \mathcal{F}_{\sigma,\lambda}(f)\in PW_r(G), \;\;\; (\sigma,\lambda) \in \widehat{M} \times \aL^*_\C$$
is a topological isomorphism between the two Fréchet spaces $C_r^\infty(G)$ and $PW_r(G)$. \qed
\end{thm}

\begin{req}
Delorme formulated the Paley-Wiener Thm.~\ref{thm:Delorme1} in terms of all cuspidal parabolic subgroups.
By Casselman's subrepresentation theorem (e.g. \cite{Wallach}, Thm.~3.8.3.), it is clear that it remains true if we restrict to the minimal parabolic subgroup $P$ (compare \cite{vandenBan}, Lem.~4.4).)
\end{req}

%Section 2
\section{Fourier transforms for (distributional) sections and its properties} \label{sect:FT}
Let $(\tau,E_\tau)$ be a finite dimensional, not necessary irreducible, representation of $K$.
We obtain a homogeneous vector bundle $\E_\tau$ over $X$, whose space $C^\infty(X,\E_\tau)$ of smooth sections is identified with the following space:
$$C^\infty(X,\E_\tau) \cong \{ f:G \stackrel{C^\infty}{\longrightarrow} E_\tau \;|\; f(gk)= \tau^{-1}(k)(f(g)), \forall g \in, k \in K\}.$$
The group $G$ acts on $C^\infty(X,\E_\tau)$ by left translations, $(g \cdot f)(g')=f(g^{-1}g'), \forall g,g' \in G$. It is not difficult to see that we have the following $G$-isomorphisms:
$$C^{\infty}(X, \E_\tau) \cong  C^\infty(G,E_\tau)^K \cong [C^{ \infty}(G) \otimes E_\tau]^K.$$
Moreover, by taking the topological linear dual of $C^\infty(X,\E_\tau)$, we obtain the space of compactly supported distributional sections:
\begin{equation} \label{eq:Distspace}
C^{-\infty}_c(X,\E_{\tilde{\tau}})= \bigcup_{r\geq 0} C^{-\infty}_r(X,\E_\tau) :=\bigcup_{r \geq 0} \{T \in C^{-\infty}(X,\E_\tau)\;|\; supp(T) \in \overline{B}_r(o)\}=C^\infty(X,\E_\tau))',
\end{equation}
where $(\tilde{\tau},E_{\tilde{\tau}})$ is the dual of the representation $(\tau,E_\tau)$.

\subsubsection*{Fourier transform in (\textcolor{blue}{Level 2})}
We want to study the reduced Fourier transform $\mathcal{F}$ on the space $[C^{\pm \infty}_c(G) \otimes E_\tau]^K \cong C^{\pm}_c(X,\E_\tau)$ by
$$\sum_{i=1}^{d_\tau} f_i \otimes v_i \mapsto \sum_{i=1}^{d_\tau} \mathcal{F}(f_i) \otimes v_i, \;\;\; f\in C^{\pm \infty}_{c}(G),$$
where $d_\tau$ denotes the dimension of $E_\tau$ and $v_i, i\in\{1, \cdots, d_\tau\}$, is a basis of $E_\tau$.
Roughly, for $r> 0$, one can deduce from Thm.~\ref{thm:Delorme1}, that
$$C^\infty_{r}(X, \E_\tau) \cong [C^\infty_{r}(G) \otimes E_\tau]^K \stackrel{\text{Thm.~\ref{thm:Delorme1}}}{\cong} [PW_r(G) \otimes E_\tau]^K,$$
where $PW_r(G)$ is Delorme's Paley-Wiener space defined in (\ref{eq:PWDelorme}).
The goal is to make $[PW_r(G) \otimes E_\tau]^K$ more explicit and then do the same study for distributions.
For this, let us study the map
\begin{eqnarray*}
C^\infty_{r}(X, \E_\tau) \ni f &\mapsto& \sum_{i=1}^{d_\tau} f_i \otimes v_i \in [C^\infty_{r}(G) \otimes E_\tau]^K\\
&\stackrel{\text{Thm.~\ref{thm:Delorme1}}}{\mapsto}& \sum_{i=1}^{d_\tau} \mathcal{F}_{\sigma,\lambda}(f_i) \otimes v_i \in [\End(H_\infty^{\sigma}) \otimes E_\tau]^K \cong H^{\sigma}_\infty \otimes \Hom_K(H^{\sigma}_\infty,E_\tau).
\end{eqnarray*}
%% Define Frob. reciprocity
Bringing the Frobenius reciprocity into play, it gives us a better description of the space $\Hom_K(H^{\sigma}_\infty,E_\tau)$. Namely, we have
\begin{eqnarray} \label{eq:Frob}
\Hom_K(H^{\sigma}_\infty,E_\tau) &\stackrel{Frob}{\cong}& \Hom_M(E_\sigma,E_\tau) \text{ defined by} \nonumber\\
\langle Frob(S)w, \tilde{v} \rangle &=& \langle w, S^*\tilde{v}(e)\rangle, \;\;\;\;\; w \in E_\sigma, \;\tilde{v} \in E_{\tilde{\tau}}, \; S^*: E_{\tilde{\tau}} \rightarrow H^{\tilde{\sigma}}_\infty.
\end{eqnarray}
Let us next compute the inverse of $Frob$.
\begin{lem}[\cite{Martin}, Lem. 2.12]\label{lem:inverseFrob}
Let $s\in \emph\Hom_M(E_\sigma,E_\tau)$ and $f\in H^{\sigma}_{\infty}$. Then, we have
$$Frob^{-1}(s)(f)=\int_K \tau(k)sf(k)\; dk. \QEDA$$
\end{lem}
\noindent
The dual of $Frob$ is given by
\begin{eqnarray} \label{eq:dualFrob}
\Hom_K(E_\tau,H^{\sigma}_\infty) &\stackrel{\widetilde{Frob}}{\cong}& \Hom_M(E_\tau,E_\sigma) \nonumber\\
\widetilde{Frob}(T)(v)&=&T(v)(e), \;\;\;\;\; v \in E_\tau
\end{eqnarray}
and for $t\in \Hom_M(E_\tau,E_\sigma)$ and $v\in E_\tau$, the inverse of $\widetilde{Frob}$ will be
\begin{equation} \label{eq:invFrobdual}
\widetilde{Frob}^{-1}(t)(v)(k)=t \tau(k^{-1})v.
\end{equation}

\noindent
Coming back to our previous computation, we get
\begin{eqnarray} \label{eq:variant1}
[\End(H_\infty^{\sigma}) \otimes E_\tau]^K 
\cong H^{\sigma}_\infty \otimes \Hom_K(H^{\sigma}_\infty,E_\tau)
&\stackrel{Frob}{\cong}& H^{\sigma}_\infty \otimes \Hom_M(E_\sigma,E_\tau) \nonumber \\
&\stackrel{(\ref{eq:Hsigmainfty})}{\cong}& C^\infty(K/M,E_\sigma \otimes \Hom_M(E_\sigma,E_\tau)) \nonumber \\
&\cong& C^\infty(K/M,\E_{\tau|_M}(\sigma)) \nonumber \\
&\cong& H^{\tau|_M(\sigma)}_\infty,
\end{eqnarray}
where 
$\E_{\tau|_M}(\sigma)$ is the $\sigma$-isotypic component of $\E_{\tau|_M}$.
Here, $\tau$ is restricted to $M$,  it is generally no more irreducible and splits into a finite direct sum
$\tau|_M=\bigoplus_{\sigma \in \hat{M}} m(\sigma,\tau)\sigma,$
where $ m(\sigma,\tau)=\dim(\Hom_M(E_\sigma,E_\tau)) \geq 0$ is the multiplicity of $\sigma$ in $\tau|_M$.
Now by taking the algebraic direct sum over all $\sigma \in \widehat{M}$, where only finitely many of them appears, we obtain
\begin{eqnarray*} 
\bigoplus_{\sigma \in \widehat{M}} [\End(H_\infty^{\sigma}) \otimes E_\tau]^K 
\stackrel{(\ref{eq:variant1})}{\cong} \bigoplus_{\sigma \in \widehat{M}} H^{\tau|_M(\sigma)}_\infty
\cong H^{\tau|_M}_\infty 
= \{f:K \stackrel{C^\infty}{\rightarrow} E_\tau \;|\; f(km)=\tau(m)^{-1}f(k)\},
\end{eqnarray*}
which can be viewed as the principal series representations corresponding to $\tau|_M$.

\begin{defn}[Fourier transform for sections over homogeneous vector bundles in (\textcolor{blue}{Level 2})] \label{defn:FTsect}
Let $g=\kappa(g)a(g)n(g) \in KAN=G$ be the Iwasawa decomposition.
For fixed $\lambda \in \aL^*_\C$ and $k\in K$, we define the function $e^\tau_{\lambda,k}$ by
\begin{eqnarray} \label{eq:exptau}
e^\tau_{\lambda,k}: G &\rightarrow& \emph\End(E_\tau) \cong E_{\tilde{\tau}} \otimes E_\tau  \nonumber \\
g &\mapsto& e^\tau_{\lambda,k}(g):=\tau(\kappa(g^{-1}k))^{-1} a(g^{-1}k)^{-(\lambda+\rho)}.
\end{eqnarray}
\begin{itemize}
\item[(a)] For $f\in C^\infty_c(X,\E_\tau)$, the Fourier transformation is given by
\begin{equation} \label{eq:FTtau}
\mathcal{F}_\tau f(\lambda,k)=\int_G e^\tau_{\lambda,k}(g) f(g) \;dg = \int_{G/K} e^\tau_{\lambda,k}(g) f(g) \;dg,
\end{equation}
where the last equality makes sense, since the integrand is right $K$-invariant.
\item[(b)] The Fourier transform for distributional section $T\in C^{-\infty}_c(X,\E_\tau)$ is defined by
$$\mathcal{F}_\tau T(\lambda,k):= \langle T, e^{\tau}_{\lambda, k} \rangle = T(e^\tau_{\lambda,k}) \in E_\tau, \;\;\; (\lambda, k) \in \aL^*_\C \times K/M.$$
\end{itemize}
\end{defn}

Note that the Fourier transform for sections has already been introduced and studied by Camporesi (\cite{Camporesi}, (3.18)). It is a direct generalization of Helgason's Fourier transform for $E_\tau=\C$.
It is not difficult to see that $\mathcal{F}_\tau f(\lambda,\cdot)$ and $\mathcal{F}_\tau T(\lambda,\cdot)$ are in $\Hol(\aL^*_\C, H^{\tau|_M}_\infty)$.
Observe that, for $k\in K$ and $g\in G$, we have, by definition
\begin{equation} \label{eq:expfctleft}
e^\tau_{\lambda,k}(g)=l_k(e^\tau_{\lambda,1}(g))=e^\tau_{\lambda,1}(k^{-1}g).
\end{equation}
This function $e^\tau_{\lambda,k}$ in Def.~\ref{defn:FTsect} can be seen as the analogous of the 'exponential' function in the definition of Fourier transform in the Euclidean case $\R^n$.
It has some interesting properties.
Note that for fixed $k\in K$, $e^\tau_{\lambda,k}(g)$ is an entire function on $\lambda \in \aL^*_\C$, since $a(g^{-1}k)^{-(\lambda+\rho)}$ is an entire function on $\lambda \in \aL^*_\C$.

\begin{prop}
Let $\tau \in \widehat{K}$, $\lambda \in \aL^*_\C$ and $k\in K$. Then, we have
\begin{equation} \label{eq:expfctgh}
e^\tau_{\lambda,k}(hg) = e^\tau_{\lambda,\kappa(h^{-1}k)}(g)a(h^{-1}k)^{-(\lambda+\rho)}, \;\;\; g,h\in G.
\end{equation}
\end{prop}

\begin{proof}
Let $h, g \in G=KAN$, then by Iwasawa decomposition, we have
\begin{eqnarray*}
hg= h \kappa(g)a(g)n(g)
&=&\kappa(h(\kappa(g))\; a(h\kappa(g))\; n(h\kappa(g)) \;a(g) \;n(g) \\
&=&\underbrace{\kappa(h\kappa(g))}_{\in K}\; \underbrace{a(h\kappa(g))\;a(g)}_{\in A}\; \underbrace{n(h\kappa(g))\;n(g)}_{\in N}.
\end{eqnarray*}
In other words, we have
$
\kappa(hg)=\kappa(h \kappa(g)a(g)n(g))=\kappa(h(\kappa(g)),$ and \\
$a(hg)=a(h \kappa(g)a(g)n(g))=a(h\kappa(g))\;a(g).$ Hence, 
\begin{eqnarray*}
e^\tau_{\lambda,k}(hg)
&\stackrel{(\ref{eq:exptau})}{=} &
\tau(\kappa(g^{-1}h^{-1}k))^{-1}a(g^{-1}h^{-1}k)^{-(\lambda + \rho)} \\
&=& \tau(\kappa(g^{-1}\kappa(h^{-1}k))^{-1}a(g^{-1}\kappa(h^{-1}k))^{-(\lambda+\rho)} a(h^{-1}k)^{-(\lambda+\rho)} \\
&\stackrel{(\ref{eq:exptau})}{=}&  e^\tau_{\lambda,\kappa(h^{-1}k)}(g)a(h^{-1}k)^{-(\lambda+\rho)}. \qedhere
\end{eqnarray*}
\end{proof}

\subsubsection*{Fourier transform in (\textcolor{blue}{Level 3}) and its properties}
Now consider an additional finite-dimensional $K$-representation
$\gamma: K \rightarrow GL(E_\gamma)$ with its associated homogeneous vector bundle $\E_\gamma$ over $X$.
It induces a mapping
\begin{equation} \label{eq:multispace}
\text{Hom}_K(E_\gamma, C^{\infty}_c(X,\E_\tau)) \longrightarrow \Hom_K(E_\gamma,\Hol(\aL^*_\C, H^{\tau|_M}_\infty)).
\end{equation}
The LHS of (\ref{eq:multispace}) can be identified with a space of functions with values in $\text{Hom}(E_\gamma, E_\tau)$, the $(\gamma,\tau)$-spherical functions:
\begin{eqnarray*}
\Hom_K(E_\gamma, C^{\infty}_c(X,\E_\tau)) 
&\cong& C^\infty_c(G,\gamma, \tau) \\
&:=&\{f : G \rightarrow \Hom(E_\gamma,E_{\tau}) \;|\; f(k_1gk_2)=\tau(k_2)^{-1}f(g)\gamma(k_1)^{-1}, \forall k_1,k_2 \in K\}.
\end{eqnarray*}
\noindent
For the RHS of (\ref{eq:multispace}), we use the Frobenius reciprocity between $K$ and $M$, by evaluating at $k=1$, and we obtain the space of functions
$\{\phi:\aL_\C^* \rightarrow \text{Hom}_M(E_\gamma,E_\tau)\}.$
Now we define the Fourier transformation $\prescript{}{\gamma}{\mathcal{F}}_\tau$ of $f \in C^{\infty}_c(G,\gamma, \tau)$.

\begin{defn}[Fourier transform in (\textcolor{blue}{Level 3})] \label{defn:FTLevel3}
With the previous notations, the Fourier transformation for $f \in C^\infty_c(G,\gamma, \tau)$ is given by
\begin{equation} \label{eq:FTgammatau3}
\prescript{}{\gamma}{\mathcal{F}}_\tau f(\lambda):= \int_G e^\tau_{\lambda,1}(g) f(g) \;dg, \;\;\;\; \lambda \in \aL^*_\C.
\end{equation}
Similar, the Fourier transformation for distributional function $T \in C^{-\infty}_c(G,\gamma, \tau)$ is defined by
$$ \prescript{}{\gamma}{\mathcal{F}}_\tau T(\lambda):= \langle T, e^\tau_{\lambda,1} \rangle.$$
\end{defn}
\noindent
Observe that
$$\tau(m) \mathcal{F}_\gamma f(\lambda)
= \int_G e^\tau_{\lambda,1}(mg) f(g) \; dg
= \int_G e^\tau_{\lambda,1}(g) f(m^{-1}g) \;dg
= \prescript{}{\gamma}{\mathcal{F}}_\tau f(\lambda) \gamma(m) \in \Hom_{M}(E_\gamma, E_\tau),$$
 same for the distributions.
Let us consider now the convolution $G$ of $f \in C^{\infty}_c(X,\E_\gamma)$ to a ($\gamma,\tau)$-spherical function $\varphi \in C^\infty_c(G,\gamma, \tau)$, which is defined by
\begin{equation} \label{eq:convusual}
(f * \varphi)(g):=\int_G \varphi(x^{-1}g)f(x) \;dx= \int_G \varphi(xg) f(x^{-1}) \; dx, \;\;\;\; g\in G.
\end{equation}
By considering the corresponding Fourier transform, we obtain the following result, which is analogous as Lem. 1.4. in (\cite{Helgason3}, Chap. 3).

\begin{prop} \label{prop:FTconv}
With the notations above, we then have that 
\begin{equation*} \label{eq:convFT}
\mathcal{F}_\tau(f * \varphi)(\lambda,k)=\prescript{}{\gamma}{\mathcal{F}}_{\tau}\varphi(\lambda) \mathcal{F}_\gamma f(\lambda,k), \;\;\; \lambda \in \aL^*_\C,k\in K.
\end{equation*}
\end{prop}

\begin{proof}
For $(\lambda,k) \in \aL^*_\C \times K$, we compute
\begin{eqnarray*}
\mathcal{F}_\tau(f* \varphi)(\lambda,k)
&\stackrel{(\ref{eq:convusual})}{=}& \int_{G \times G} e^\tau_{\lambda,k}(g) \varphi(\underbrace{x^{-1}g}_{=:h})f(x) \; dx\;dg \\
&\stackrel{\text{Fubini's thm.}}{=}& \int_{G} \Big( \int_G e^\tau_{\lambda,k}(xh) \varphi(h) \;dh \Big) f(x) \; dx \\
&\stackrel{(\ref{eq:exptau})}{=}&  \int_G \Big(  \int_G e^\tau_{\lambda,\kappa(x^{-1}k)}(h) a(x^{-1}k)^{-(\lambda+\rho)} \varphi(h) \;dh \Big) f(x) \; dx \\
&\stackrel{(\ref{eq:expfctleft})}{=}&  \int_G \Big(  \int_G e^\tau_{\lambda,1}(\underbrace{\kappa(x^{-1}k)^{-1}h}_{=:g})  \varphi(h) \;dh \Big)a(x^{-1}k)^{-(\lambda+\rho)} f(x) \; dx \\
&=& \int_G \Big(  \int_G e^{\tau}_{\lambda,1}(g) \textcolor{blue}{\varphi(\kappa(g^{-1}k)g)} \; dg \Big) a(x^{-1}k)^{-(\lambda + \rho)} f(x) \; dx\\
&=& \int_G \Big( \int_G e^{\tau}_{\lambda,1}(g) \textcolor{blue}{\varphi(g)} \;dg \Big) \textcolor{blue}{\gamma(\kappa(x^{-1}k))^{-1}} a(x^{-1}k)^{-(\lambda+\rho)} f(x) \; dx \\
&=&\prescript{}{\gamma}{\mathcal{F}}_{\tau}\varphi(\lambda) \mathcal{F}_\gamma f(\lambda,k).
\qedhere
\end{eqnarray*}
\end{proof}

\begin{req} \label{req:leftconv}
\begin{itemize}
\item[(a)] If $\gamma = \tau$, then we have
$\mathcal{F}_\tau(f * \varphi)(\lambda,k)=\prescript{}{\tau}{\mathcal{F}}_{\tau}\varphi(\lambda) \mathcal{F}_\tau f(\lambda,k),$
for $f\in C^{\infty}_c(X,\E_\tau)$ and a spherical function $\varphi \in C^\infty_c(G,\tau,\tau)$.
\item[(b)]
In a smiliar way, one can define the left convolution for scalar valued-function $\varphi \in C^\infty_c(G)$. In fact, we know that, for $f\in C_c^\infty(X,\E_\tau)$ and $g\in G$, we have
\begin{eqnarray*}
\mathcal{F}_\tau(l_g f)(\lambda,k)
= \int_G e^\tau_{\lambda,k}(x) l_g f(x) \; dx
&=& \int_G e^\tau_{\lambda,k}(gh) f(h) \; dh \\
&\stackrel{(\ref{eq:expfctgh})}{=}& a(g^{-1}k)^{-(\lambda + \rho)} \int_G e^\tau_{\lambda,\kappa(g^{-1}k)}(h)  f(h) \; dh\\
&\stackrel{(\ref{eq:repK})}{=}&(\pi_{\tau,\lambda}(g) {\mathcal{F}}_{\tau}f(\lambda, \cdot))(k).
\end{eqnarray*}
Hence, we can deduce for $\varphi \in C_c^\infty(G)$:
\begin{equation} \label{eq:leftconv}
\mathcal{F}_\tau(\varphi * f)(\lambda,k)= 
(\pi_{\tau,\lambda}(\varphi) {\mathcal{F}}_{\tau}f(\lambda, \cdot))(k).
\end{equation}
\item[(c)] Analogously as for smooth compactly functions (\ref{eq:convusual}), we define the convolution for distributions $T \in C^{-\infty}_c(X,\E_\tau)$ by
\begin{equation*} \label{eq:convdist}
(T * \varphi)(g):= T(l_g \varphi^\vee)= \langle T, l_g \varphi^\vee \rangle, \;\;\; g \in G, \varphi \in C^\infty_c(G, \tau,\tau),
\end{equation*}
where $\varphi^\vee \in C^\infty_c(X,\E_{\tilde{\tau}}) \otimes E_\tau$ is given by $\varphi^{\vee}(g):=\varphi(g^{-1}), g \in G$. Then, the obtained results can be applied for distributions as well.
\end{itemize}
\end{req}

\noindent
Now, for positive $\epsilon>0$, take a $K$-conjugation invariant open neighbourhood $U_\epsilon \subset B_\epsilon(0)$ so that $\bigcap_{\epsilon > 0} U_\epsilon = \{0\}$,
and for $\epsilon_1 < \epsilon_2$, we have $U_{\epsilon_1} \subset U_{\epsilon_2}.$
Consider a scalar-valued positive function $\tilde{\eta}_{\epsilon} \in C^\infty_c(U_\epsilon) \subset C^\infty_c(G)$ in $G$ satisfying
\begin{equation} \label{eq:epsilonUe}
\int_{U_\epsilon} \tilde{\eta}_{\epsilon}(g) \; dg=1.
\end{equation}
Note that $\tilde{\eta}_{\epsilon}$ cannot be $K \times K$-invariant. Let us construct from this an endomorphism function $\eta_\epsilon \in C^\infty(G,\tau, \tau)$ by
\begin{equation} \label{eq:etag}
\eta_\epsilon(g):= \int_{K\times K} \tilde{\eta}_{\epsilon}(k_1 g k_2) \tau(k_1 k_2) \; dk_1 \;dk_2, \;\;\; g \in G.
\end{equation}
Then, we get the following observation.

\begin{cor} \label{cor:FTconv}
For each $\epsilon > 0$, let $\eta_\epsilon \in C_c^\infty(G,\tau, \tau)$ be the $K \times K$-invariant endomorphism function (\ref{eq:etag}).
Then, its Fourier transform $\prescript{}{\tau}{\mathcal{F}}_{\tau}\eta_\epsilon$ converges uniformly on compact sets $C$ on $\aL^*_\C$ to the identity map:
$$\prescript{}{\tau}{\mathcal{F}}_{\tau}\eta_\epsilon(\lambda) \rightarrow \emph \Id, \;\;\; \lambda \in C$$
when $\epsilon \rightarrow 0$.
\end{cor}

\begin{proof}
Consider $\eta_\epsilon \in C^\infty(G,\tau, \tau)$, then for $g\in G$:
\begin{eqnarray*}
\eta_\epsilon(g)= \int_{K} \int_{K} \tilde{\eta}_{\epsilon}(k_1 g k_2) \tau(k_1 k_2) \; dk_1 \;dk_2
= \int_K \int_K \tilde{\eta}_{\epsilon}(k_1 g l k_1^{-1}) \tau(l)\; dk_1\; dl
=\int_K \overline{\eta}_\epsilon (gl) \tau(l) \; dl,
\end{eqnarray*}
where we did a change of variable and set $\overline{\eta}_\epsilon(g):= \int_K \tilde{\eta}_\epsilon(k_1 gk_1^{-1}) \; dk_1$.
Here, $\tilde{\eta}_\epsilon \in C^\infty_c(U_\epsilon)$ as above (\ref{eq:epsilonUe}).
By computing its Fourier transform, we obtain, for $\lambda \in \aL^*_\C$
\begin{eqnarray*}
\prescript{}{\tau}{\mathcal{F}}_{\tau}(\eta_\epsilon)(\lambda)
\stackrel{(\ref{eq:FTgammatau3})}{=} \int_G e^\tau_{\lambda,1}(g) \eta_\epsilon(g) \;dg
&=&  \int_G \Big(\int_{K} e^\tau_{\lambda,1}(g) \overline{\eta}_\epsilon (gl) \tau(l) \; dl \Big)\;dg \\
&\stackrel{(\ref{eq:7})}{=}& \int_G \Big(\int_K e^\tau_{\lambda,1}(gl)\overline{\eta}_\epsilon (gl)\; dl \Big)\;dg \\
&=& \int_G e^\tau_{\lambda,1}(g)\overline{\eta}_\epsilon (g) \; dg\\
&=& \int_{U_\epsilon} \overline{\eta}_\epsilon(g) (e^\tau_{\lambda,1}(g) -\Id) dg +\Id.
\end{eqnarray*}
Now, consider a compact set $C$ on $\aL^*_\C$ and $\delta >0$, then there exists $\epsilon >0$ such that
\begin{center}
$|e^\tau_{\lambda,1}(g)-\Id| < \delta$ for $g\in U_\epsilon, \lambda \in C$.
\end{center}
Thus, this implies that $\prescript{}{\tau}{\mathcal{F}}_{\tau}\eta_\epsilon$ converges uniformly on compact sets to $\text{Id}$, when $\epsilon$ converges to~$0$.
\end{proof}

Furthermore, consider an non-zero linear $G$-invariant differential operator between sections over homogeneous vector bundles
\begin{equation} \label{eq:invDO}
D: C^\infty(X,\E_\tau) \longrightarrow C^\infty(X,\E_\gamma)
\end{equation}
such that $D(g \cdot f)=g \cdot (Df),$ for all $g\in G, f \in C^\infty(X,\E_\tau)$.
Denote by $\DD_G(\E_\tau,\E_\gamma)$ the vector space of all these $G$-invariant differential operators on sections. We get the following relation.

\begin{prop} \label{prop:Qexpfct}
Let $Q \in \DD_G(\E_{\tilde{\tau}}, \E_{\tilde{\gamma}})$ be an invariant linear differential operator. 
Then, we have
\begin{equation} \label{eq:step2}
Q e^\tau_{\lambda,k} =(Qe^\tau_{\lambda,1}(1)) \circ e^\gamma_{\lambda,k}, \;\;\; \lambda \in \aL^*_\C, k \in K.
\end{equation}
\end{prop}

\begin{proof} Let us first consider the case $k=1$. We then have for $g\in G=NAK$:
\begin{equation} \label{eq:7}
e^\tau_{\lambda,1}(g)=e^\tau_{\lambda,1}(nak_1)=a^{\lambda+\rho} \tau (k_1)=a^{\lambda+\rho} e^\tau_{\lambda,1}(k_1), \;\;\; n \in N, a \in A, k_1 \in K.
\end{equation}
In particular, for $n_1a_1 \in NA$
\begin{eqnarray*}
l_{(n_1 a_1)^{-1}} e^\tau_{\lambda,1}(nak_1)
=e^\tau_{\lambda,1}(n_1a_1nak_1)
=e^\tau_{\lambda,1}(n_1(a_1 na_1^{-1})a_1 ak_1)
&\stackrel{(\ref{eq:7})}{=}& a_1^{\lambda+\rho}a^{\lambda+\rho}\tau(k_1) \\
&=& a_1^{\lambda+\rho} e^\tau_{\lambda,1}(g).
\end{eqnarray*}
Hence, since $Q$ is linear and $G$-invariant, we obtain that
\begin{equation} \label{eq:Q12}
 l_{(n_1a_1)^{-1}} (Qe^\tau_{\lambda,1}(g))=Q( l_{(n_1a_1)^{-1}} e^\tau_{\lambda,1}(g))=Q(a_1^{\lambda+\rho} e^\tau_{\lambda,1}(g))=a_1^{\lambda+\rho}Q(e^\tau_{\lambda,1}(g))
\end{equation}
and by setting $g=k_1=1$, we have
\begin{equation} \label{eq:Q}
Qe^\tau_{\lambda,1}(n_1a_1)\stackrel{(\ref{eq:Q12})}{=}a_1^{\lambda+\rho}Qe^\tau_{\lambda,1}(1).
\end{equation}
Therefore, since $e^\tau_{\lambda,1} \in C^\infty(X,\E_{\tilde{\tau}}) \otimes E_\tau \subset C^\infty(G,\End(E_\tau))$, we have that $Qe^\tau_{\lambda,1} \in  C^\infty(X,\E_{\tilde{\gamma}}) \otimes E_\tau \subset C^\infty(G,\Hom(E_\gamma,E_\tau)).$
Therefore, for $g=n_1a_1k_2 \in G$, we can conclude that
\begin{eqnarray} \label{eq:Q1}
&&Qe^\tau_{\lambda,1}(n_1a_1k_2)
=Qe^\tau_{\lambda,1}(n_1a_1)\gamma(k_2)
\stackrel{(\ref{eq:Q})}{=} a_1^{\lambda+\rho}(Qe^\tau_{\lambda,1}(1))\gamma(k_2)
\stackrel{(\ref{eq:7})}{=} (Qe^\tau_{\lambda,1}(1))e^\gamma_{\lambda,1}(n_1a_1k_2). \nonumber \\
&&
\end{eqnarray}
Now for general $k\in K$, we observe that
$e^\tau_{\lambda,k}=l_ke^\tau_{\lambda,1}$. Hence
$$ Qe^\tau_{\lambda,k}=Q(l_ke^\tau_{\lambda,1})\stackrel{(\ref{eq:Q1})}{=}l_k (Qe^\tau_{\lambda,1}(1))e^\gamma_{\lambda,1}=(Qe^\tau_{\lambda,1}(1)) \circ e^\gamma_{\lambda,k}.$$
Thus, we get the desired result.
\end{proof}

\section{Delorme's intertwining conditions and some examples} \label{sect:DelormeIntw}
We study Delorme's intertwining conditions $(D.a)$ in Def.~\ref{defn:intwcond} and determine the  intertwining conditions in (\textcolor{blue}{Level 2}) and (\textcolor{blue}{Level 3}) induced by them.
To do this, we firstly need some preparations.
 In the previous Section~\ref{sect:FT}, we have seen that the identification (\ref{eq:variant1}).
Let us now take a closer look.
Consider the Frobenius-reciprocity (\ref{eq:Frob}) with its dual (\ref{eq:dualFrob}) and define the map
$$
I: \bigoplus_{\sigma \in \hat{M}} H^{\sigma}_\infty \otimes \Hom_K(H^{\sigma}_\infty,E_\tau) \longrightarrow H^{\tau|_M}_\infty$$
by $I(\alpha)=d_\sigma \sum_{i=1}^{m(\tau,\sigma)} s_i\alpha_i,$ for $\alpha = \sum_{i=1}^{m(\tau,\sigma)} \alpha_i \otimes S_i\in H^{\sigma}_\infty \otimes \Hom_K(H^{\sigma}_\infty,E_\tau)$,
where $s_i = Frob(S_i)$ runs a basis through $\Hom_M(E_\sigma,E_\tau)$, for all $i$.
Here, $m(\tau,\sigma)$ stands for the dimension of the multiplicity space $\Hom_K(H^{\sigma}_\infty,E_\tau)$.
For $T \in \Hom_K(E_\tau,H^{\sigma}_\infty)$, let
$$\langle \alpha, T \rangle := \sum_{i=1}^{m(\tau,\sigma)} \alpha_i \cdot \Tr_{\tau}(S_i\circ T).$$
Now, by using the identification 
$
[\End(H_\infty^{\sigma}) \otimes E_\tau]^K
\stackrel{j}{\cong} 
 H^{\sigma}_\infty \otimes \Hom_K(H^{\sigma}_\infty,E_\tau),
$
we can define the map
\begin{equation} \label{eq:mapJ}
J: \bigoplus_{\sigma \in \hat{M}} [\End(H^{\sigma}_\infty) \otimes E_\tau]^K  \longrightarrow H^{\tau|_M}_\infty
\end{equation}
by $J=I \circ j.$
In addition, for $\beta= \sum_{i=1}^{d_\tau} \beta_i \otimes v_i \in \bigoplus_{\sigma \in \hat{M}} [\End(H^{\sigma}_\infty) \otimes E_\tau]^K$ and $T \in \Hom_K(E_\tau,H^{\sigma}_\infty)$, let 
\begin{equation} \label{eq:bracketsT}
 \langle \beta, T \rangle := \sum_{i=1}^{d_\tau} \beta_i \circ T(v_i) \in H_\infty^\sigma,
\end{equation}
where $\{v_i, i=1,\dots, d_\tau\}$ runs a vector basis of $E_\tau$.
One checks that $\langle \beta, T \rangle= \langle j(\beta), T \rangle$.

\begin{prop} \label{prop:Level123} 
With the previous notations, let $f:=\sum_i^{d_{\tau}} f_i \otimes v_i \in C^\infty_c(X,\E_\tau)$. Denote by $\mathcal{F}_\tau(f)$ its Fourier transform in $H^{\tau|_M}_\infty$ given in (\ref{eq:FTtau}).\\ 
Then, for $T \in \emph \Hom_K(E_\tau,H^\sigma_\infty)$ and $t=\widetilde{Frob}^{-1}(T)\in \emph \Hom_M(E_\tau,E_\sigma)$, we obtain
\begin{itemize}
\item[(1)] $\langle \alpha, T \rangle = t \circ I(\alpha),$ 
\item[(2)]
$\langle \mathcal{F}_{\sigma,\lambda}(f), T \rangle = t \circ \mathcal{F}_\tau f(\lambda, \cdot) \in H^{\sigma,\lambda}_{\infty},$ for $\lambda \in \aL^*_\C,$
\item[(3)] 
$\mathcal{F}_\tau f(\lambda, \cdot) = J (\bigoplus_{\sigma\in \widehat{M}} \mathcal{F}_{\sigma,\lambda}(f)),$ for $\lambda \in \aL^*_\C$.
\end{itemize}
\end{prop}

\begin{proof}
\begin{itemize}
\item[(1)] It is sufficient to prove it for only one summand in $\alpha$, hence let $\alpha=\alpha_1 \otimes S$. For $T=\widetilde{Frob}(t)\in \Hom_K(E_\tau,H^\sigma_\infty)$ and $S=Frob(s) \in \Hom_K(H^\sigma_\infty,E_\tau)$, we thus obtain
\begin{eqnarray*}
\langle \alpha, T \rangle 
= \alpha_1 \Tr_\tau (S \circ T)
&\stackrel{\text{Lem.}~\ref{lem:inverseFrob}+(\ref{eq:invFrobdual})}{=}& \alpha_1 \Tr_\tau \Big(v \mapsto \int_K \tau(k) s\circ t(\tau(k^{-1}))v\; dk\Big), \; v \in E_\tau \\
&=&  \alpha_1 \Tr_\tau\Big( \int_K \tau(k) s \circ t \tau(k^{-1}) \; dk\Big) \\
&=& \alpha_1 \Tr_\tau (s \circ t) \\
&=&\Tr_\sigma (t \circ s)  \alpha_1.
\end{eqnarray*}
Since $\sigma \in \widehat{M}$ is irreducible and $t\circ s \in \End_M(E_\sigma)$,
by Schur's lemma, we have that $t\circ s = \lambda \cdot \Id$, for some $\lambda \in \C$ and thus $\Tr_{\sigma}(t \circ s)= \lambda$.
Hence $\langle \alpha, T \rangle = (t \circ s)(\alpha_1) = t(I(\alpha)).$
\item[(2)] By computation, we obtain
\begin{eqnarray*}
\langle \mathcal{F}_{\sigma,\lambda}(f), T \rangle 
= \sum_{i=1}^{d_\tau} \mathcal{F}_{\sigma,\lambda}(f_i) \circ T(v_i) 
&\stackrel{(\ref{eq:invFrobdual})}{=}& \sum_{i=1}^{d_\tau} \mathcal{F}_{\sigma,\lambda}(f_i)(t\tau(\cdot)(v_i))\\
&\stackrel{\text{Def.}~\ref{def:FTDelorme}}{=}& \sum_{i=1}^{d_\tau} \int_G f_i(g) \pi_{\sigma,\lambda}(g) (t\tau(\cdot)(v_i)) \; dg\\
&=& \sum_{i=1}^{d_\tau} \int_G f_i(g) (\pi_{\sigma,\lambda}(g) \varphi_i)(\cdot) \; dg.\\
\end{eqnarray*}
In the last line, we set $\varphi_i(k):=t\tau(k^{-1})(v_i)$, for $k\in K$.
Fix $k\in K$, by applying (\ref{eq:repK}), we have
$(\pi_{\sigma,\lambda}(g) \varphi_i)(k)=a(g^{-1}k)^{-(\lambda+\rho)} \varphi(\kappa(g^{-1}k))^{-1}.$

Thus,
\begin{eqnarray*}
 \sum_{i=1}^{d_\tau} \int_G f_i(g) t \tau(\kappa(g^{-1}k))^{-1} a(g^{-1}k)^{-(\lambda+\rho)} v_i \; dg
&=&  \sum_{i=1}^{d_\tau} \int_G f_i(g) t e^\tau_{\lambda,k}(g) v_i \; dg \\
&=& t \circ \int_G  \sum_{i=1}^{d_\tau} e^\tau_{\lambda,k}(g) f_i(g) v_i \; dg \\
&=& t \circ \int_G  e^\tau_{\lambda,k}(g) f(g) \; dg = t \circ \mathcal{F}_\tau f(\lambda,k).
\end{eqnarray*}

\item[(3)] By rewritting (1) and (2) in the following way:
\begin{itemize}
\item[(1')] $\Tr_\tau (I^{-1}(\alpha) \circ T) = t \circ \alpha$,
\item[(2')] $\Tr_\tau ( \mathcal{F}_{\sigma,\lambda}(f) \circ T)=t \circ \mathcal{F}_\tau f(\lambda, \cdot),$
\end{itemize}
we get that
$$\Tr_\tau (J^{-1}(\mathcal{F}_\tau f(\lambda, \cdot)) \circ T)= \Tr_\tau (I^{-1}(\mathcal{F}_\tau f(\lambda, \cdot)) \circ T) \stackrel{(1')}{=}  t \circ \mathcal{F}_\tau f(\lambda, \cdot) \stackrel{(2')}{=} \Tr_\tau( \mathcal{F}_{\sigma,\lambda}(f) \circ T).$$
By taking only the $\sigma$-component of $\bigoplus_{\sigma \in \widehat{M}} [\End(H^\sigma_\infty) \otimes E_\tau]^K$, we have that 
the parining in $\Tr_\sigma$ is non-degenerate, thus
$J^{-1}(\mathcal{F}_\tau f(\lambda, \cdot))= \bigoplus_{\sigma\in \widehat{M}} \mathcal{F}_{\sigma,\lambda}(f)$. \qedhere
\end{itemize}
\end{proof}

\noindent
We first study what happens to Delorme's intertwining condition $(D.a)$ if we tensor it with $E_\tau$ and take $K$-invariants.

\begin{defn} \label{defn:intwcondLevel1'}
Consider $\tau \in \widehat{K}.$
\begin{itemize}
\item[(1)] We say that a function $$\phi \in 
\prod_{\sigma \in \widehat{M}}[ \emph \Hol(\aL^*_\C, \emph \End(H^\sigma_\infty)) \otimes E_\tau]^K
\cong \bigoplus_{\sigma \subset \tau|_M}  \emph \Hol(\aL^*_\C, [\emph \End(H^\sigma_\infty) \otimes E_\tau]^K)$$
satisfies the intertwining condition, if for each $\tilde{v} \in E_{\tilde{\tau}}$:
$$\langle \phi, \tilde{v} \rangle_\tau \in  \prod_{\sigma \in \widehat{M}} \emph \Hol(\aL^*_\C, \emph \End(H^\sigma_\infty))$$
satisfies the intertwining condition in Def.~\ref{defn:intwcond}.
\end{itemize}
\end{defn}

\begin{prop} \label{prop:Level01}
Let $\phi \in \prod_{\sigma \in \widehat{M}}[ \emph \Hol(\aL^*_\C, \emph \End(H^\sigma_\infty)) \otimes E_\tau]^K$ as in Def.~\ref{defn:intwcondLevel1'} and $(\xi,W)$ the intertwining data defined in Def.~\ref{defn:intwcond}.
\begin{itemize}
\item[$(D.1)$] Then, $\phi$ satisfies the intertwining condition (1) of Def.~\ref{defn:intwcondLevel1'} if, and only if, for each  intertwining datum $(\xi, W)$ and $T\in \emph \Hom_K(E_\tau,W) \subset \emph \Hom_K(E_\tau,H_\xi)$, the induced element $\phi_\xi \in \emph [\emph \End(H_\xi) \otimes E_\tau]^K$ satisfies
$$\langle \phi_\xi, T \rangle \in W.$$
\end{itemize}
\end{prop}

\begin{proof}
For each $i\in \{1,\dots,d_\tau\}$, consider $f_i \in \End(H_\xi)$ so that for each intertwining datum $(\xi,W)$, we have $f_i(W) \subseteq W$.
Consider
$$\phi_\xi=\sum_{i=1}^{d_\tau} f_i \otimes v_i  \in [\End(H_\xi) \otimes E_\tau]^K$$
as in Thm.~\ref{thm:equivintcond}.
It is sufficient to show that for each $i$ and $T \in \Hom_K(E_\tau,W)$, we have
$f_i \circ T \in W$ if, and only if, $\langle \phi_\xi, T \rangle \in W,$ $\forall T \in \Hom_K(E_\tau,W).$

The right implication is obvious. By using the definition of the brackets $\langle \cdot, \cdot \rangle$ as in (\ref{eq:bracketsT}), we have
$$\langle \phi_\xi, T \rangle = \sum_{i=1}^{d_\tau} f_i \circ T(v_i) \in W$$
since for $v_i \in E_\tau$, $T(v_i) \in W \subset H_\xi$.

For the left implication, write $f_i =\langle \phi_\xi, \tilde{v}_i \rangle_\tau,$ for all $i\in \{1, \dots, d_\tau\}$, where $\tilde{v}_i$ runs a dual basis of $\E_{\tilde{\tau}}$.
Consider the mapping $A_{ij} \in \End(E_\tau)$ such that $v_i \mapsto v_j$ and $v_k \mapsto 0, \; k \neq i$. Then, for all $ i,j\in \{1, \dots, d_\tau\}$, we have
$$
f_i \circ T(v_j)
= \langle \phi_\xi,  T(v_j) \cdot \tilde{v}_i \rangle
= \langle \phi_\xi,  T \circ A_{ij} \rangle
= \langle \phi_\xi,  p_K(T \circ A_{ij}) \rangle,
$$
where $p_K: \Hom(E_\tau, W) \rightarrow \Hom_K(E_\tau,W)$ is the orthogonal projection. 
Note that $T(v_j) \cdot \tilde{v}_i \in \Hom(E_\tau,W)$, for all $i,j$.
By setting, now in the last line $T'_{ij}:=p_K(T \circ A_{ij})$, we get that $\langle \phi_\xi, T'_{ij}\rangle \in W$. Thus, for all $i\in \{1, \dots, d_\tau\}$, we have $f_i \circ T \in W$.
\end{proof}

Next, we state the intertwining condition in (\textcolor{blue}{Level 2}) and (\textcolor{blue}{Level 3}) induced from Delorme's intertwining condition (D.a), more presciely (1) in Def.~\ref{defn:intwcondLevel1'}.

\begin{defn}[Intertwining conditions in (\textcolor{blue}{Level 2}) and (\textcolor{blue}{Level 3})]
\label{defn:intwcondLevel23}
Let $\tau, \gamma \in \hat{K}$ and consider the map $J$ defined in (\ref{eq:mapJ}).
\begin{itemize}
\item[(2)] We say that a function $\psi \in \emph \Hol(\aL^*_\C, H^{\tau|_M}_\infty)$ satisfies the intertwining condition, if 
$$J^{-1}\psi \in \bigoplus_{\sigma \subset \tau|_M}  \emph \Hol(\aL^*_\C, [\emph \End(H^\sigma_\infty) \otimes E_\tau]^K)$$
satisfies the intertwining condition $(1)$ in Def.~\ref{defn:intwcondLevel1'}.
\item[(3)] We say that a function $\varphi \in \emph \Hol(\aL^*_\C, \emph \Hom_M(E_\gamma,E_\tau))$ satisfies the intertwining condition, if for all $w\in E_\gamma$:
$$(\lambda,k) \mapsto \varphi(w)(\lambda,k):=\varphi(\lambda)\gamma(k^{-1})w \in  \emph \Hol(\aL^*_\C, H^{\tau|_M}_\infty)$$
satisfies the above intertwining condition $(2)$.
\end{itemize}
\end{defn}

\noindent
We now want to make the intertwining conditions more explicit. 
Let us first introduce some notations.
We define
$$\Hom_M(E_\tau,E_\sigma)^{\lambda}_{(m)}:= \Hol(\aL^*_\C, \Hom_M(E_\tau,E_\sigma))/m_\lambda^{m+1} \;  \Hol(\aL^*_\C, \Hom_M(E_\tau,E_\sigma))$$
as in (\ref{eq:succderv}), similarly for $H^{\tau|_M,\lambda}_{\infty,(m)}$.
For $\tau \in \widehat{K}$ and each intertwining datum $(\xi,W)$,
consider
\begin{eqnarray} \label{eq:DtauW}
D^\tau_W&:=&\{t \in \bigoplus_{i=1}^s \Hom_M(E_\tau,E_{\sigma_i})^{\lambda_i}_{(m_i)}  \;|\;
 T=\widetilde{Frob}^{-1}(t) \in \Hom_K(E_\tau,W) \subset \Hom_K(E_\tau, H_\xi)\} \nonumber \\
&\subset&\bigoplus_{i=1}^s \Hom_M(E_\tau,E_{\sigma_i})^{\lambda_i}_{(m_i)}.
\end{eqnarray}
Write by $\overline{\Xi}$ the set of all 2-tuples $(\lambda,m)$ with $\lambda \in \aL^*_\C$ and $m\in \N_0$ and we define the map
$$ \Xi \longrightarrow \overline{\Xi}, \;\;\; \xi=(\sigma,\lambda,m) \mapsto \overline{\xi}=(\lambda,m).$$
For $s\in \N$ and $\xi \in \Xi^s$, we have the corresponding element $\overline{\xi}\in \overline{\Xi}^s$.

\begin{thm}[Intertwining conditions in the three levels]\label{thm:equivintcond}
With the notations above, we then have:
\begin{itemize}
\item[(D.2)] (\textcolor{blue}{Level 2})
Then, $\psi \in \emph \Hol(\aL^*_\C, H^{\tau|_M}_\infty)$ satisfies the intertwining condition (2) of Def.~\ref{defn:intwcondLevel23} if, and only if, for each intertwining datum $(\xi,W)$ and each non-zero $t=(t_1,t_2, \dots, t_s) \in D^\tau_W$, the induced element
$\psi_{\overline{\xi}} \in  \bigoplus_{i=1}^s H^{\tau|_M,\lambda_i}_{\infty,(m_i)}=:H^{\tau|_M}_{\overline{\xi}}$ satisfies
$$ t \circ \psi_{\overline{\xi}}=(t_1 \circ \psi_1, \dots, t_2 \circ \psi_s) \in W.$$
\end{itemize}
\begin{itemize}
\item[(D.3)] (\textcolor{blue}{Level 3})
Then, $\varphi \in \emph \Hol(\aL^*_\C, \emph \Hom_M(E_\gamma,E_\tau))$  satisfies the intertwining condition (3) of Def.~\ref{defn:intwcondLevel23} if, and only if, for each intertwining datum $(\xi,W)$ and each non-zero $t=(t_1,t_2, \dots, t_s) \in D^\tau_W$, the induced element
$\varphi_{\overline{\xi}} \in  \bigoplus_{i=1}^s \emph \Hom_M(E_\gamma,E_\tau)^{\lambda_i}_{(m_i)} =:H^{\gamma,\tau}_{\overline{\xi}}$ satisfies
$$  t \circ \varphi_{\overline{\xi}}=(t_1 \circ \varphi_1, \dots, t_2 \circ \varphi_s) \in D^\gamma_W.$$
\end{itemize}
\end{thm}

\begin{proof}%[Proof of Thm.~\ref{thm:equivintcond}]
We obtain directly the equivalence between $(D.1)$ for $J^{-1} \psi$ and $(D.2)$ for $\psi$ by applying the Frobenius reciprocity, Prop.~\ref{prop:Level01} and Prop.~\ref{prop:Level123} (2).

Concerning $(D.2)$ for $\varphi (w)$ $\iff (D.3)$ for $\varphi$, one implication is trivial. 
For the other one, we have, by the inverse dual Frobenius reciprocity, that
$$W \ni t\circ \psi_{\overline{\xi}} = t \circ \widetilde{Frob}^{-1}(\varphi_{\overline{\xi}})(w)(k)\stackrel{(\ref{eq:dualFrob})}{=} t \circ
\varphi_{\overline{\xi}} \circ \gamma(k^{-1})w, \;\; \forall t \in D^\tau_W,$$
for $w\in E_\gamma$ and $k \in K$.
This means that 
$\widetilde{Frob}^{-1}(t \circ \varphi_{\overline{\xi}})(w) \in \Hom_K(E_\gamma, W)$
and hence by applying the dual Frobenius-reciprocity
$ \Hom_K(E_\gamma, W) \stackrel{\widetilde{Frob}}{\cong} D^\gamma_W$,
 this implies that $t\circ \varphi_{\overline{\xi}} \in D^\gamma_W$.
\end{proof}

\begin{exmp} \label{exmp:equivintwcond}
\begin{itemize}
\item[(a)] Consider $s=1$ and $m=0$. 
Let  $\xi:=(\sigma,\lambda,0) \in \Xi$ and $W \subset H^{\sigma,\lambda}_\infty$.
Consider $D^\tau_W \subset \Hom_M(E_\tau,E_\sigma)$ as in Thm.~\ref{thm:equivintcond}.
Then, we have the following intertwining conditions in the corresponding levels:
\begin{itemize}
\item[(D.2a)] (\textcolor{blue}{Level 2}) 
For each intertwining datum $(\xi,W)$ and $0 \neq t\in D^\tau_W$, we have
$$t \circ \psi(\lambda,\cdot) \in W.$$
Note that for each $\overline{\xi} \in  \overline{\Xi}$, the induced element $\psi_{\overline{\xi}}=
\psi(\lambda,\cdot)$.
\item[(D.3a)] (\textcolor{blue}{Level 3})
For each intertwining datum $(\xi,W)$ and $0 \neq t \in D^\tau_W$, we have
$$t \circ \varphi(\lambda) \in D^\gamma_W.$$
Note that for each $\overline{\xi} \in  \overline{\Xi}$, the induced element
$\varphi_{\overline{\xi}}=\varphi(\lambda).$
\end{itemize}
\item[(b)] Consider now $s=2$ and $m_1=m_2=0$.
Let $L: H^{\sigma_1,\lambda_1}_\infty \longrightarrow  H^{\sigma_2,\lambda_2}_\infty$ be an intertwining operator between the two principal series representations.
Let $\xi:=((\sigma_1,\lambda_1,0),(\sigma_2,\lambda_2,0)) \in \Xi^2$ and
$W=graph(L) \subset  H^{\sigma_1,\lambda_1}_\infty \oplus H^{\sigma_2,\lambda_2}_\infty$.
Moreover, define
$l^\tau: \Hom_M(E_\tau,E_{\sigma_1}) \longrightarrow \Hom_M(E_\tau,E_{\sigma_2})$ by
$$l^\tau(t)(v)=L(t \tau(\cdot)^{-1}v)(e)$$
for $v\in E_\tau$ and $t\in \Hom_M(E_\tau,E_{\sigma_1})$.
Then
\begin{eqnarray*}
D^\tau_W=\{(t_1,t_2)\;|\; t_2=l^\tau(t_1)\} 
&=& \{(t,l^\tau(t))\;|\; t \in \Hom_M(E_\tau,E_{\sigma_1})\} \\
&\subset& \Hom_M(E_\tau,E_{\sigma_1}) \oplus \Hom_M(E_\tau,E_{\sigma_2}).
\end{eqnarray*}
In this situation, we have the following intertwining conditions.
\begin{itemize}
\item[(D2.b)] (\textcolor{blue}{Level 2}) For each $t \in \Hom_M(E_\tau,E_{\sigma_1})$, we have for $\psi(\lambda_i, \cdot) \in H^{\tau|_M, \cdot}_\infty, i=1,2$
\begin{equation} \label{eq:Lintwcond2}
 L( t\circ \psi(\lambda_1, \cdot))= l^\tau(t) \circ \psi(\lambda_2, \cdot).
\end{equation}
\item[(D3.b)] (\textcolor{blue}{Level 3})
For each $t \in \Hom_M(E_\tau,E_{\sigma_1})$, we have for $\varphi(\lambda_i) \in \Hom_M(E_\gamma,E_\tau), i=1,2$
\begin{equation} \label{eq:Lintwcond3}
 l^\gamma(t \circ \varphi(\lambda_1)) = l^\tau(t) \circ \varphi(\lambda_2).
\end{equation}
\end{itemize}
\end{itemize}
\end{exmp}	

\section{Topological Paley-Wiener theorem for sections}
The Paley-Wiener space for sections over homogeneous vector bundles is defined as follows.
\begin{defn}[Paley-Wiener space for sections in (\textcolor{blue}{Level 2}) and (\textcolor{blue}{Level 3})]
\label{def:PWsect}
\noindent
\begin{itemize}
\item[(a)] For $r> 0$, let $PW_{\tau,r}(\aL^*_\C \times K/M)$ be the space of sections
$\psi \in C^\infty(\aL^*_\C \times K/M, \E_{\tau|_M})$
be such that
\begin{itemize}
\item[$(2.i)$] the section $\psi$ is holomorphic in $\lambda \in \aL^*_\C$, i.e. $\psi \in \emph \Hol(\aL^*_\C, H^{\tau|_M}_\infty).$
\item[$(2.ii)_r$] (growth condition) for all $Y \in \U(\mathfrak{k})$ and $N \in \N_0$, there exists a constant $C_{r,N,Y}>0$ such that
$$||l_Y\psi(\lambda,k)||_{E_\tau} \leq C_{r,N,Y} 
(1+|\lambda|^2)^{-N} e^{r|\emph \Repart(\lambda)|},\;\;\; k \in K,$$
where $||\cdot||_{E_\tau}$  denotes the norm on finite-dimensional vector space $E_\tau$ (for convenience, we often denotes it by $|\cdot|$).
\item[$(2.iii)$] (intertwining condition) (D.2) from Thm.~\ref{thm:equivintcond}.
\end{itemize}

\item[(b)] By considering an additional $K$-type, let $\prescript{}{\gamma}{PW}_{\tau,r}(\aL^*_\C)$ be the space of functions
$$\aL^*_\C \ni \lambda \mapsto \varphi(\lambda) \in \emph\Hom_M(E_\gamma,E_{\tau})$$
be such that
\begin{itemize}
\item[$(3.i)$] the function $\varphi$ is holomorphic in $\lambda \in \aL^*_\C$.
\item[$(3.ii)_r$] (growth condition) for all $N \in \N_0$, there exists a constant $C_{r,N}>0$ such that
$$||\varphi(\lambda)||_{\text{op}} \leq C_{r,N} 
(1+|\lambda|^2)^{-N} e^{r|\emph \Repart(\lambda)|},$$
where $||\cdot||_{\text{op}}$ denotes the operator norm on $\emph \Hom_M(E_\gamma,E_\tau)$.
\item[$(3.iii)$] (intertwining condition) (D.3) from Thm.~\ref{thm:equivintcond}.
\end{itemize}
\end{itemize}
\end{defn}

The inequalities provide semi-norms $||\cdot||_{r,N,Y}$ (resp. $||\cdot||_{r,N}$) on $PW_{\tau,r}(\aL^*_\C \times K/M)$ (resp. $\prescript{}{\gamma}{PW}_{\tau,r}(\aL^*_\C)$) and made
the vector space $PW_{\tau,r}(\aL^*_\C \times K/M)$ (resp. $\prescript{}{\gamma}{PW}_{\tau,r}(\aL^*_\C)$) to Fréchet space, e.g. one can compare Lem. 10 of Delorme  \cite{Delorme}.

\noindent
Combining Delorme's Paley-Wiener Thm.~\ref{thm:Delorme1} with the above identifications and observations, we obtain a Paley-Wiener theorem in (\textcolor{blue}{Level 2}) and (\textcolor{blue}{Level 3}).

\begin{thm}[Topological Paley-Wiener theorem for sections in (\textcolor{blue}{Level 2}) and (\textcolor{blue}{Level 3})] \label{thm:PWsect}
Let $(\tau, E_\tau)$ be a $K$-representation with associated homogeneous vector bundle $\E_\tau$. For $r> 0$, then the Fourier transform
$$C_r^\infty(X,\E_\tau) \ni \psi \mapsto \mathcal{F}_{\tau}(\psi)(\lambda, k)\in PW_{\tau,r}(\aL^*_\C \times K/M), \;\;\; (\lambda,k) \in \aL^*_\C \times K$$
is a topological isomorphism between $C_r^\infty(X,\E_\tau)$ and $PW_{\tau,r}(\aL^*_\C \times K/M)$.\\
Moreover, by considering an additional $K$-representation $(\gamma,E_\gamma)$ with associated homogeneous vector bundle $\E_\gamma$, then the Fourier transform
$$ C^\infty_r(G,\gamma, \tau) \ni \varphi \mapsto \prescript{}{\gamma}{\mathcal{F}}_{\tau}(\varphi)(\lambda) \in \prescript{}{\gamma}{PW}_{\tau,r}(\aL^*_\C), \;\;\; \lambda \in \aL^*_\C$$
is a topological isomorphism between $C^\infty_r(G,\gamma, \tau)$ and $\prescript{}{\gamma}{PW}_{\tau,r}(\aL^*_\C)$. \QEDA
\end{thm}

Furthermore, by taking the union of all $r> 0$, the Paley-Wiener space $PW_{\tau}(\aL^*_\C \times K/M)$ is defined as
$$PW_{\tau}(\aL^*_\C \times K/M):= \bigcup_{r> 0} PW_{\tau,r}(\aL^*_\C \times K/M)$$
similar for $\prescript{}{\gamma}{PW}_{\tau}(\aL^*_\C)$. Equip $PW_{\tau}(\aL^*_\C \times K/M)$ and $\prescript{}{\gamma}{PW}_{\tau}(\aL^*_\C)$ with the inductive limit topology (compare the next Sect.~\ref{sect:PWS}).
Hence, by the above result (Thm.~\ref{thm:PWsect}), we also have a linear topological Fourier transform isomorphism from $C^\infty_c(X,\E_\tau)$ (resp. $C^\infty_c(G,\gamma, \tau)$) onto $PW_{\tau}(\aL^*_\C \times K/M)$ (resp. $\prescript{}{\gamma}{PW}_{\tau}(\aL^*_\C)$).

\section{On topological Paley-Wiener-Schwartz theorem for sections and its proof} \label{sect:PWS}

\subsubsection*{Distributional sections and their corresponding topology}
In (\ref{eq:Distspace}), we already introduced the vector space $C^{-\infty}_c(X,\E_\tau)$ by taking the taking the topological linear dual of $C^\infty(X,\E_{\tilde{\tau}})$.
We provide $C^{-\infty}_c(X,\E_\tau)$ with the \textit{strong dual topology}.
Actually, we know that $C^\infty(X,\E_{\tilde{\tau}})$ is a Fréchet space with semi-norm
\begin{equation} \label{eq:seminormf}
||h||_{\Omega,Y} := \sup_{g\in \Omega} |l_Y h(g)|, \;\;\;\; h \in C^\infty(X,\E_{\tilde{\tau}}),
\end{equation}
where $Y\in \U(\g)$ and $\Omega$ is a compact subset of $G$.
Furthermore, a subset $B \subset C^\infty(X,\E_{\tilde{\tau}})$ is called bounded, if
for each compact $\Omega \subset G$ and $Y\in \U(\g)$ there exists a constant $C_{\Omega,Y} >0$ such that
$\sup_{\varphi \in B} ||\varphi||_{\Omega,Y} \leq C_{\Omega,Y}.$
Shortly, every semi-norm is bounded on $B$.\\
The strong dual topology on $C^{-\infty}_c(X,\E_\tau)$ is a locally convex topology vector space given by the semi-norm system
\begin{equation} \label{eq:seminormTB}
p_B(T):=||T||_B = \sup_{\varphi \in B} |T(\varphi)| = \sup_{\varphi \in B} |\langle T,\varphi \rangle|, \;\;\; T \in C^{-\infty}_c(X,\E_\tau),
\end{equation}
where $B$ belongs to the family of all bounded subsets of $C^\infty(X,\E_{\tilde{\tau}})$.
Similarly, we equip $C^{-\infty}(X,\E_\tau)=(C^\infty_c(X,\E_{\tilde{\tau}}))'$ with the strong dual topology.
As an immediate consequence of theses dualities, the topologies on $C^{-\infty}_c(X,\E_\tau)$ and $C^{-\infty}(X,\E_\tau)$ induce the same topology on the space of distributions supported in a fixed compact subset $\Omega$ of $G$ (\cite{vandenBanDist}, Sect. 14). 
For example, one can take $\Omega=\overline{B}_r(o).$

A subset $B' \subset C^{-\infty}_c(X,\E_\tau)$ is bounded in the strong dual topology, if for each bounded $B \subset C^\infty(X,\E_{\tilde{\tau}})$, we have
\begin{equation} \label{eq:boundedT}
\sup_{T\in B'} p_B(T) = \sup_{T\in B', \varphi \in B} |T(\varphi)| < \infty.
\end{equation}
Since, by Schaefer (\cite{Schaeffer}, Cor.~1.6, p.~127), we know that all such sets $B'$ are equicontinuous, this means that there exist a continuous semi-norm $p$ on $C^\infty(X,\E_{\tilde{\tau}})$ and a constant $C>0$ such that
$$B' \subset \{T\in C^{-\infty}_c(X,\E_\tau) \;|\; |T(\varphi)| \leq C p(\varphi), \forall \varphi \in C^\infty(X,\E_{\tilde{\tau}})\}.$$
Let $Y_1,\dots,Y_n$ be a basis of $\g$, then for a multi-index $\alpha \in \N^n_0$, we set
$Y_\alpha :=Y^{\alpha_1}_1 \cdots Y^{\alpha_n}_n \in \U(\g).$
We may assume that the semi-norm $p$ has the form
\begin{equation} \label{eq:seminormp}
p(\varphi)=  \sum_{|\alpha| \leq m} ||\varphi||_{\Omega,\alpha} \stackrel{(\ref{eq:seminormf})}{=} \sum_{|\alpha| \leq m} \sup_{g \in \Omega} |l_{Y_\alpha}\varphi(g)|, \;\;\;\;\; \varphi \in C^\infty(X,\E_{\tilde{\tau}}),\forall \alpha
\end{equation}
for some $m \in \N_0$ and compact $\Omega \subset G$.

It is interesting to notice that $C^\infty(X,\E_{\tilde{\tau}})$ is a reflexive Fréchet space, even a Montel space, that is, it is reflexive and a subset is bounded if, and only if, it is relatively compact (\cite{Schaeffer}, p. 147). \\
Thus, since $C^{-\infty}_c(X,\E_\tau)$ is the strong dual space of a Montel space $C^\infty(X,\E_{\tilde{\tau}})$, we can deduce by Cor.~1 in (\cite{Schaeffer}, p.154)
that $C^{-\infty}_c(X,\E_\tau)$ is a bornological space,
that is a locally convex space on which each semi-norm $p_B$, which is bounded on bounded subsets, is continuous (\cite{Schaeffer}, Chap.2.8, p. 61).\\
This observation leads us to the following general result, which will play an imporant role in the proof of the Paley-Wiener-Schwartz theorem.
For bornological spaces, bounded linear maps are continuous (\cite{Schaeffer}, Thm.~8.3., p.~62), hence, we obtain the following.

\begin{lem} \label{lem:Montelspace}
Let $W$ be any locally convex topological vector space and consider a linear map
$$A:C^{-\infty}_c(X,\E_\tau) \rightarrow W.$$
Then $A$ is continuous if, and only if, $A(B')$ is bounded in $W$, for every bounded subset $B' \subset C^{-\infty}_c(X,\E_\tau)$. \QEDA
\end{lem}

\noindent
Let $Y_1,\dots,Y_k$ be a basis of $\U(\kk)$, then for a multi-index $\alpha \in \N^k_0$, we have 
$Y_\alpha:=Y^{\alpha_1}_1 \cdots Y^{\alpha_k}_k \in \U(\kk).$
Now we are in the position to define Paley-Wiener-Schwartz space for sections. 

\begin{defn}[Paley-Wiener-Schwartz space for sections in (\textcolor{blue}{Level 2}) and (\textcolor{blue}{Level 3})]
\label{def:PWSspace}
\noindent
\begin{itemize}
\item[(a)] For $r> 0$, let $PWS_{\tau,r}(\aL^*_\C \times K/M)$ be the space of sections
$\psi \in C^\infty(\aL^*_\C \times K/M, \E_{\tau|_M})$
be such that
\begin{itemize}
\item[$(2.i)$] the section $\psi$ is holomorphic in $\lambda \in \aL^*_\C$, i.e. $\psi \in \emph\Hol(\aL^*_\C, H^{\tau|_M}_\infty)$.
\item[$(2.iis)_r$] (growth condition) for all multi-indices $\alpha$, there exist $N \in \N_0$ and a positive constant $C_{r,N,\alpha}$ such that
$$||l_{Y_\alpha} \psi(\lambda,k)||_{E_\tau} \leq C_{r,N,\alpha} (1+|\lambda|^2)^{N+\frac{|\alpha|}{2}} e^{r|\emph \Repart(\lambda)|},\;\;\; k \in K.$$
\item[$(2.iii)$] (intertwining condition) (D.2) from Thm.~\ref{thm:equivintcond}.
\end{itemize}

\item[(b)] By considering an additional $K$-type, let $\prescript{}{\gamma}{PWS}_{\tau,r}(\aL^*_\C)$ be the space of functions
$$\aL^*_\C \ni \lambda \mapsto \varphi(\lambda) \in \emph\Hom_M(E_\gamma,E_{\tau})$$
be such that
\begin{itemize}
\item[$(3.i)$] the function $\varphi$ is holomorphic in $\lambda \in \aL^*_\C$.
\item[$(3.iis)_r$] (growth condition) there exist $N\in \N_0$ and a positive constant $C_{r,N}$ such that
$$
||\varphi(\lambda)||_{\text{op}} \leq C_{r,N} (1+|\lambda|^2)^{N} e^{r|\emph \Repart(\lambda)|}.
$$
%where $||\cdot||_{\text{op}}$ denotes the operator norm on the corresponding space.
\item[$(3.iii)$] (intertwining condition) (D.3) from Thm.~\ref{thm:equivintcond}.
\end{itemize}
\end{itemize}
\end{defn}

For all $r \geq 0$ and $N\in \N_0$, we consider
 $$PWS_{\tau,r,N} := \{\psi \in PWS_{\tau,r}(\aL^*_\C \times K/M) \;|\; ||\psi||_{r,N,\alpha}< \infty, \forall \alpha \},$$
with semi-norms
\begin{equation*} \label{eq:seminormPWS}
||\psi||_{r,N,\alpha} := \sup_{\lambda \in \aL^*_\C, \; k \in K/M} (1+|\lambda|^2)^{-(N+\frac{|\alpha|}{2})} e^{-r|\Repart(\lambda)|}||l_{Y_\alpha} \psi(\lambda,k)||_{E_\tau},\;\; \forall \alpha, k \in K.
\end{equation*}
This gives $PWS_{\tau,r,N}$ the structure of a Fréchet space.
We set
$PWS_{\tau}(\aL^*_\C \times K/M):=\bigcup_{r\geq0} \bigcup_{N\in \N_0} PWS_{\tau,r,N}(\aL^*_\C \times K/M)$ and equip it with the locally convex inductive limit topology.
It is the finest locally convex topology on $PWS_{\tau}(\aL^*_\C \times K/M)$ such that all the embeddings
$PWS_{\tau,r,N} \stackrel{i_{r,N}}{\hookrightarrow} PWS_{\tau}(\aL^*_\C \times K/M)$
are continuous.
Furthermore, this topology is characterized by the following property.
A linear map $$A: PWS_{\tau} \rightarrow W,$$
where $W$ is any locally convex space, is continuous if, and only if, all the maps
$$PWS_{\tau,r,N}  \stackrel{i_{r,N}}{\hookrightarrow} PWS_{\tau} \stackrel{A}{\longrightarrow} W$$
are continuous, i.e., $A \circ i_{r,N}$ are continuous.
The exactly same procedure, can be done for $\prescript{}{\gamma}{PWS}_{\tau}(\aL^*_\C)$.

\noindent
We are now in the position to state the main theorem.
\begin{thm}[Topological Paley-Wiener-Schwartz theorem for sections] \label{thm:PWSsect}
\noindent
\begin{itemize}
\item[(a)]Let $(\tau,E_\tau)$ be a $K$-representation with associated homogeneous vector bundle $\E_\tau$.\\
Then, for each $r\geq0$, the Fourier transform $\mathcal{F}_\tau$ is a linear
bijection between the two spaces $C^{-\infty}_r(X,\E_\tau)$ and the Paley-Wiener-Schwartz space $PWS_{\tau,r}(\aL^*_\C \times K/M)$.
Moreover, it is a linear topological isomorphism from $C^{-\infty}_c(X,\E_\tau)$ onto $PWS_{\tau}(\aL^*_\C \times K/M)$.

\item[(b)] Similarly, if we consider an additional $K$-representation $(\gamma,E_\gamma)$ with associated homogeneous vector bundle $\E_\gamma$.
Then, the Fourier transform $\prescript{}{\gamma}{\mathcal{F}}_{\tau}$ is a linear
bijection between the two spaces
$C^{-\infty}_r(G,\gamma, \tau)$ and $\prescript{}{\gamma}{PWS}_{\tau,r}(\aL^*_\C)$, for each $r\geq 0$, and a linear topological isomorphism from $C^{-\infty}_c(G,\gamma, \tau)$ onto $\prescript{}{\gamma}{PWS}_{\tau}(\aL^*_\C)$.
\end{itemize}
\end{thm}

\begin{req}
Delorme proved in his paper (\cite{Delorme}), the Paley-Wiener theorem in (\textcolor{blue}{Level 1}) for Hecke algebra
\begin{equation} \label{eq:Heckealg}
\mathcal{H}(G,K):=C^{-\infty}_{r=0}(G)_K \cong \U(\g) \otimes_{\U(\kk)} C^\infty(K)_K,
\end{equation}
which consists of all $K\times K$-finite distributions on $G$ supported by $K \subset G.$
\end{req}

\subsubsection*{Harish-Chandra inversion and Plancherel Theorem for sections}
In order to prove Thm.~\ref{thm:PWSsect}, we need the Harish-Chandra Plancherel inversion formula for sections over homogeneous vector bundles. 

\begin{thm}[Plancherel Theorem for sections, \cite{Camporesi}, Thm. 3.4 \& Thm. 4.3]  \label{thm:HCinvSect}
Let $\mathcal{Q}$ be a complete set of representatives of association classes of cuspidal parabolic subgroups
$Q=M_QA_QN_Q$ with $Q \supset P=MAN$ and  $A_Q\subset A$. 
We have $\aL^*=\aL^*_Q \oplus \aL^*_{M_Q}$.\\
Then, there exists a finite set $A^\tau_Q \subset \aL^*_{M_Q} \subset \aL^*$ and for $\nu \in  A^\tau_Q$, there exists an analytic function of at most polynomial growth
$$\mu_\nu^Q : i\aL^*_Q \longrightarrow \emph \End_M(E_\tau)$$ 
such that for each $f\in C^\infty_c(X,\E_\tau)$, we have
\begin{eqnarray*}
f(e)&=&\sum_{Q \in \mathcal{Q}} \sum_{\nu \in A^\tau_Q} \int_{i\aL^*_Q} \int_K \tau(k) \mu_\nu^Q (\lambda) \mathcal{F}_\tau(f)(\nu+\lambda, k) \;dk\; d\lambda. \QEDA
\end{eqnarray*}
\end{thm}

Note that $A^\tau_P=\{0\}$. 

\begin{cor} \label{cor:HCinvSect}
With the notations above, let $f\in C^\infty_c(X,\E_\tau)$ and $\varphi \in C^\infty_c(X,\E_{\tilde{\tau}})$.
Then
\begin{equation} \label{eq:HCPformula}
\int_G \langle \varphi(g), f(g) \rangle_\tau \; dg= 
\sum_{Q \in \mathcal{Q}} \sum_{\nu\in A^\tau_Q} \int_{i\aL^*_Q} \int_K 
\langle \mathcal{F}_{\tilde{\tau}}(\varphi)(-\nu-\lambda,k) ,\mu_\nu^Q(\lambda)\mathcal{F}_\tau(f)(\nu+\lambda,k)\rangle_\tau \;dk\; d\lambda.
\end{equation}
\end{cor}

\begin{proof}
Let $\{\tilde{v}_i, i=1, \dots, d_{\tau}\}$ be a vector basis of $E_{\tilde{\tau}}$.
We write $\varphi=\sum_{i=1}^{d_\tau} \varphi_i \cdot \tilde{v}_i$ with $\varphi_i \in C^\infty_c(G)$.
For $h\in C^\infty_c(G)$, we set
$h^{\vee}(g):=h(g^{-1}).$
Then
$$ \int_G \langle \varphi(g),f(g) \rangle \; dg= \sum_{i=1}^{d_\tau} \langle (\varphi^{\vee}_i* f)(e), \tilde{v}_i \rangle,$$
where we used the usual convolution defined in (\ref{eq:convusual}).
Note that $h* f= l(h)f$, where $l$ is the (left) regular representation of $G$ on $C^\infty_c(X,\E_\tau)$.
By the $G$-equivariance of the Fourier transform, we have by (\ref{eq:leftconv}):
$\mathcal{F}_\tau (h * f)(\lambda, k)= \pi_{\tau,\lambda}(h)(\mathcal{F}_\tau (f)(\lambda, \cdot))(k).$
By applying Thm.~\ref{thm:HCinvSect}, we obtain for all $i \in \{1, \dots, d_{\tau}\}$
$$
\langle \tilde{v}_i, (\varphi^{\vee}_i * f)(e) \rangle
= \sum_{Q,\nu} \int_{i\aL^*_Q} \int_K \langle \tilde{v}_i, \tau(k) \mu_\nu^Q(\lambda) \pi_{\tau, \nu+\lambda}(\varphi^{\vee}_i) (\mathcal{F}_\tau(f)(\nu+\lambda, \cdot))(k) \rangle \; dk \; d\lambda.$$
Using that $\mu_\nu^Q$ commutes with $\pi_{\tau, \nu+\lambda}$ and that integration over $K$ gives a $G$-equivariant pairing between $H^{\tau,\nu+\lambda}_\infty$ and $H^{\tilde{\tau}, -(\nu+\lambda)}_\infty$, we obtain that the $K$-integral equals
\begin{eqnarray*}
\int_K &\langle \tilde{\tau}(k^{-1}) \tilde{v}_i,& \pi_{\tau, \nu +\lambda}(\varphi_i^{\vee}) \mu_\nu^Q(\lambda) (\mathcal{F}_\tau(f)(\nu+\lambda, \cdot))(k) \rangle \; dk\\
&=& \int_K \langle (\pi_{\tilde{\tau},-(\nu+\lambda)}(\varphi_i)\tilde{\tau}(\cdot)^{-1} \tilde{v}_i)(k), \;\; \mu_\nu^Q(\lambda) \mathcal{F}_\tau(f)(\nu+\lambda,k) \rangle \; dk.
\end{eqnarray*}
Now
\begin{eqnarray*}
(\pi_{\tilde{\tau},-(\nu+\lambda)}(\varphi_i)\tilde{\tau}(\cdot)^{-1} \tilde{v}_i)(k)
&=& \int_G \varphi_i(g) \tilde{\tau}(\kappa(g^{-1}k))^{-1} a(g^{-1}k)^{\nu+\lambda-\rho} \tilde{v}_i \; dg \\
&=& \int_G \varphi_i(g) e^{\tilde{\tau}}_{-(\nu+\lambda),k}(g) \tilde{v}_i \; dg.
\end{eqnarray*}
The sum over all $i$ equals to $\mathcal{F}_{\tilde{\tau}}(\varphi)(-(\nu+\lambda),k)$.
Combining all the previous formulas, we obtain the corollary.
\end{proof}

\subsubsection*{Proof of the topological Paley-Wiener-Schwartz Thm.~\ref{thm:PWSsect}} \label{subsect:PWSproof}
For $r\geq 0$, let us frist provide the bijection between the vector spaces $C^{-\infty}_r(X,\E_\tau)$ and $PWS_{\tau,r}(\aL^*_\C \times K/M)$.

\begin{prop} \label{lem:isoPWS}
Consider a $K$-representation $(\tau,E_\tau)$.
\begin{itemize}
\item[(a)] Let  $T\in C^{-\infty}_c(X,\E_\tau)$ such that its Fourier transform $\mathcal{F}_\tau(T)=0$, then $T=0$.
\item[(b)] For $r \geq 0$ and $\tilde{T} \in PWS_{\tau,r}(\aL^*_\C \times K/M)$, there exists $T\in C^{-\infty}_r(X,\E_\tau)$ such that $\tilde{T}=\mathcal{F}_\tau(T)$.
\item[(c)] For $r \geq 0$, let $T\in C^{-\infty}_r(X,\E_\tau)$, then $\mathcal{F}_\tau(T) \in PWS_{\tau,r}(\aL^*_\C \times K/M)$.
\end{itemize}
\end{prop}

\begin{proof}
For each $\epsilon > 0$, consider $\eta_\epsilon \in C^\infty(G, \tau,\tau)$ with compact support in the closed ball $\overline{B}_\epsilon(o)$ as in Cor.~\ref{cor:FTconv}.
Let $T \in C^{-\infty}_c(X,\E_\tau)$ be a distribution, then
$$T_\epsilon:=T * \eta_\epsilon \in C^\infty_c(X,\E_\tau).$$
Moreover, by using the same arguments as in the proof of Cor.~\ref{cor:FTconv}, we have that
$T_\epsilon \stackrel{\epsilon \rightarrow 0}{\longrightarrow} T \text{ (weakly)}.$
Hence, by the Paley-Wiener Thm.~\ref{thm:PWsect}, this implies that 
$ \mathcal{F}_\tau(T_\epsilon) \in PW_{\tau}(\aL^*_\C \times K/M)$.
Note that $\mathcal{F}_\tau(T_\epsilon)$ is holomorphic on $\lambda \in \aL^*_\C$ and it satisfies the conditions $(2.i)$ and $(2.ii)_r$ of Def.~\ref{def:PWsect}.
Furthermore, by Prop.~\ref{prop:FTconv}, we have
\begin{equation} \label{eq:FTconvnep}
\mathcal{F}_\tau(T_\epsilon)(\lambda,k)=\prescript{}{\tau}{\mathcal{F}}_{\tau}(\eta_\epsilon)(\lambda)\mathcal{F}_\tau(T)(\lambda,k), \;\;\;\; (\lambda, k) \in \aL^*_\C \times K/M.
\end{equation}
Due to Cor.~\ref{cor:FTconv}, $\prescript{}{\tau}{\mathcal{F}}_{\tau}(\eta_\epsilon)$ converges uniformly on compact subsets of $\aL^*_\C$ to the identity map, whenever $\epsilon$ tends to $0$. Hence,
$\lim_{\epsilon \rightarrow 0} \mathcal{F}_\tau(T_\epsilon)=\mathcal{F}_\tau(T) $
uniformly on compact sets on~$\aL^*_\C$. 
\begin{itemize}
\item[(a)]  Now assume that $\mathcal{F}_\tau(T)=0$. 
By (\ref{eq:FTconvnep}), we have that $\mathcal{F}_\tau(T_\epsilon)=0$. 
By applying the Paley-Wiener Thm.~\ref{thm:PWsect}, this implies that $T_\epsilon =0$.
Hence, since $T_\epsilon \stackrel{\epsilon \rightarrow 0}{\longrightarrow} T$ weakly, we have that $T=0$.
%This means that $T \mapsto \mathcal{F}_\tau(T)$ is injective on $C^{-\infty}_c(X,\E_\tau)$. 
%%
\item[(b)] Consider $\psi \in PWS_{\tau,r}(\aL^*_\C \times K/M)$.
For each $\epsilon > 0$ and $h\in C^\infty_c(X,\E_{\tilde{\tau}})$, let $T_\epsilon$ be the functional given by
\begin{eqnarray} \label{eq:Te1}
T_\epsilon(h)
&:=&  \sum_{Q \in \mathcal{Q}} \sum_{\nu \in A^\tau_Q} \int_{i\aL^*_Q} \int_K \langle 
 \mathcal{F}_{\tilde{\tau}} (h)(-\nu-\lambda,k) \;,\; \nonumber \\
&&\;\;\;\;\;\;\;\;\;\;\;\;\;\;\;\;\;\;\;\;\;\;\; \mu^Q_\nu(\lambda) \prescript{}{\tau}{\mathcal{F}}_{\tau} (\eta_\epsilon)(\nu+\lambda)\psi(\nu+\lambda,k)\rangle \; dk \; d\lambda
\end{eqnarray}
under the same notations introduced in Thm.~\ref{thm:HCinvSect}.
Notice that, since $\supp(\eta_\epsilon) \subset \overline{B}_\epsilon(o)$ and $\psi$ satisfies the 'slow' growth condition $(2.iis)_r$ of Def.~\ref{def:PWSspace}, for all $r \geq 0$, this implies that for each multi-index $\alpha \in \N_0$ and $N \in \N_0$, there exists a constant $C_{r,N,\alpha}>0$ such that
\begin{equation} \label{eq:psigrowth}
|l_{Y_\alpha} \prescript{}{\tau}{\mathcal{F}}_{\tau} (\eta_\epsilon)(\lambda)\psi(\lambda,k)| \leq C_{r,N,\alpha}(1+|\lambda|^2)^{-N}e^{(r+\epsilon)|\Repart(\lambda)|}, \;\;\; (\lambda,k) \in \aL^*_\C \times K.
\end{equation} 
In addition, for each intertwining datum $(\xi,W)$,
the induced operator
$(\prescript{}{\tau}{\mathcal{F}}_{\tau}(\eta_\epsilon)\psi)_{\overline{\xi}}=\prescript{}{\tau}{\mathcal{F}}_{\tau}(\eta_\epsilon)_{\overline{\xi}}\psi_{\overline{\xi}} \in H_\infty^{\tau|_M}$
satisfies the intertwining condition $(3.iii)$ of Def.~\ref{def:PWsect}.
In fact, for $t\in D^\tau_W$, we have
$t \circ \prescript{}{\tau}{\mathcal{F}}_{\tau}(\eta_\epsilon)_{\overline{\xi}} \in D^\tau_W$ and since $\psi \in  PWS_{\tau}(\aL^*_\C \times K/M)$, this implies that
$$(t\circ \prescript{}{\tau}{\mathcal{F}}_{\tau}(\eta_\epsilon)_{\overline{\xi}}) \circ \psi_{\overline{\xi}} \in W.$$
Therefore, by the Paley-Wiener Thm.~\ref{thm:PWsect}, we have that
$\prescript{}{\tau}{\mathcal{F}}_{\tau}(\eta_\epsilon)\psi$ is the Fourier transform of a unique function $f_\epsilon \in C^\infty_c(X,\E_{\tau})$, i.e.,
$$ \mathcal{F}_{\tau}(f_\epsilon) :=\prescript{}{\tau}{\mathcal{F}}_{\tau}(\eta_\epsilon)\psi.$$
On the other side, by (\ref{eq:Te1}) and Cor.~\ref{cor:HCinvSect}, we have $T_\epsilon=f_\epsilon$.
By (\ref{eq:psigrowth}), we have that $\supp(T_\epsilon) \subset \overline{B}_{r+\epsilon}(o)$.
Thus, by Cor.~\ref{cor:FTconv}, this implies that
\begin{eqnarray} \label{eq:Th}
T_\epsilon(h) \stackrel{\epsilon \rightarrow 0}{\longrightarrow} T(h):=\sum_{Q \in \mathcal{Q}} \sum_{\nu \in A^\tau_Q} \int_{i\aL^*_Q} \int_K \langle \mathcal{F}_{\tilde{\tau}} (h)(-\nu-\lambda,k),\;
\mu^Q_\nu(\lambda) \psi(\nu+\lambda,k)
 \rangle \; dk \; d\lambda \nonumber \\
\end{eqnarray}
and thus $\supp(T) \subset \overline{B}_{r}(o).$
Note that $\mu^Q_\nu$ has at most polynomial growth, thus $T$ is well-defined and continuous.
Since $T$ is compactly supported, we can set $h:=e^\tau_{\lambda,k}$. 
In conclusion, we have found a distribution $T\in C^{-\infty}_r(X,\E_\tau)$ such that
\begin{eqnarray*}
\mathcal{F}_\tau(T)(\lambda,k)
=T(e^\tau_{\lambda,k})
\stackrel{(\ref{eq:Th})}{=}  \lim_{\epsilon \rightarrow 0} T_\epsilon(e^\tau_{\lambda,k})
=\lim_{\epsilon \rightarrow 0} \mathcal{F}_{\tau}(f_\epsilon)(\lambda,k)
&=&\lim_{\epsilon \rightarrow 0} \prescript{}{\tau}{\mathcal{F}}_{\tau}(\eta_\epsilon)(\lambda)\psi(\lambda,k) \\
&=&\psi(\lambda,k).
\end{eqnarray*}

\item[(c)]  
Let us check that for $r \geq 0$, $\mathcal{F}_\tau(T) \in PWS_{\tau,r}(\aL^*_\C \times K/M)$.
This means that we need to verify that the Fourier transform of $T\in C^{-\infty}_r(X,\E_\tau)$ satisfies the conditions $(2.i)-(2.iii)$ of Def.~\ref{def:PWsect}.\\
The condition $(2.i)$ is immediate.
Concerning the intertwining condition $(2.iii)$, in order to show that for each intertwining datum $(\xi,W)$ and
$t\in D^\tau_W$, we have
$$t \circ (\mathcal{F}_\tau(T))_{\overline{\xi}}
 \in 
W \subseteq H_{\xi},$$
we will use a similar convolution argument as above, except that now we are interested to the convolution on the left instead on the right. 
For each $\epsilon > 0$, let $\delta_\epsilon \in C^\infty_c(G)$ be a delta-sequence such that 
$\lim_{\epsilon \rightarrow 0} \delta_\epsilon= \delta_0.$ 
Hence, $\lim_{\epsilon \rightarrow 0} \delta_\epsilon * T=T$, for $T \in C^{-\infty}_r(X,\E_\tau)$.
Moreover, for all representations $(\pi_{\tau,\lambda},H)$ with Fréchet space $H$ and $v \in H$, we have 
$\pi_{\tau,\lambda}(\delta_\epsilon)v \stackrel{\epsilon \rightarrow 0}{\longrightarrow} v.$
By taking the Fourier transform on $\delta_\epsilon * T \in C^{\infty}_r(X,\E_\tau)$, we first prove that for each intertwining datum $(\xi,W)$ and $t \in D^\tau_W$:
$$
\lim_{\epsilon \rightarrow 0} (t \circ  \mathcal{F}_\tau(\delta_\epsilon * T)_{\overline{\xi}}) \in W.
$$
In fact, we have
\begin{eqnarray*} %\label{eq:leftconvpi}
t \circ \mathcal{F}_\tau(\delta_\epsilon * T)_{\overline{\xi}}
 &\stackrel{\text{Remark}~\ref{req:leftconv}}{=}& t \circ (\pi_{\tau,\cdot}(\delta_\epsilon)\mathcal{F}_{\tau}(T))_{\overline{\xi}}  \\
&=& (\dots, t_i \circ (\pi_{\tau,\lambda_i}(\delta_\epsilon) \mathcal{F}_{\tau}(T)(\lambda_i,\cdot))_{(m_i)}, \dots) \\
&=& (\dots, t_i \circ \pi_{\xi_i}(\delta_\epsilon) \mathcal{F}_{\tau}(T)_{\overline{\xi}_i}, \dots) \\
&=&( \dots, \pi_{\xi_i}(\delta_\epsilon) (t_i \circ {\mathcal{F}}_{\tau}(T)_{\overline{\xi}_i}), \dots) \\
&=& \pi_\xi(\delta_\epsilon) ( t \circ {\mathcal{F}}_{\tau}(T)_{\overline{\xi}}) \in W,
\end{eqnarray*}
where 
$(\pi_{\sigma_1,\lambda_1}^{(m_1)}(\delta_\epsilon), \dots, \pi_{\sigma_s,\lambda_s}^{(m_s)}(\delta_\epsilon))
 = \pi_\xi(\delta_\epsilon) \in W \subset H_{\xi}.$
Hence, by taking $\epsilon \rightarrow 0$ and since $W$ is closed, we obtain that 
$t \circ (\mathcal{F}_\tau(T))_{\overline{\xi}}  \in W$.\\
It remains to check that $\mathcal{F}_\tau(T)$ statisfies the slow growth condition $(2.iis)_r$. 
Fix $r \geq 0$.
We need to show that for each multi-index $\alpha$, there exist $N\in \N_0$ and a constant $C_{r,N,\alpha}>0$ such that
$$|l_{Y_\alpha} \mathcal{F}_\tau(T)(\lambda,k)| \leq C_{r,N,\alpha} (1+|\lambda|^2)^{N+\frac{|\alpha|}{2}} e^{r|\Repart(\lambda)|}.$$
Note that $l_{Y_\alpha} \mathcal{F}_\tau(T) = \mathcal{F}_\tau(l_{Y_\alpha} T)$.
Let $T\in C^{-\infty}_r(X,\E_\tau)$ be a distribution of order $m \in \N_0$.
Write $X_\beta \in \U(\mathfrak{n})$ and $H_\gamma \in \U(\aL)$ for all multi-indices $\beta, \gamma$.
Since $G/K \cong NA$ and $\U(\mathfrak{n} \oplus \aL) \cong \U(\mathfrak{n}) \U(\aL)$, 
then, there exists a constant $C>0$ such that
\begin{equation} \label{eq:HelgasonIneq}
|T(h)| \leq C  \sum_{|\beta|+|\gamma| \leq m} \sup_{g\in \overline{B}_r(o)} |(l_{X_\beta} (l_{H_\gamma} h))(g)|, \;\;\;\; \forall h \in C^\infty(X,\E_{\tilde{\tau}}).
\end{equation}
Next, we want to apply it to $h=e^\tau_{\lambda,1}$. We observe that
$$ l_{Y_\alpha}\mathcal{F}_\tau(T)(\lambda,k)
= \mathcal{F}_\tau(l_{Y_\alpha}T)(\lambda,k)
=l_{Y_\alpha}T(e^\tau_{\lambda,k})
\stackrel{(\ref{eq:expfctleft})}{=}(l_{Y_\alpha}T)(l_k e^\tau_{\lambda,1})
=(l_{k^{-1}}l_{Y_\alpha}T)(h).$$
Thus, $l_{k^{-1}}l_{Y_\alpha}T$ is a distribution of order $m+|\alpha|$.
Applying (\ref{eq:HelgasonIneq}) to $(l_{k^{-1}}l_{Y_\alpha}T)(h)$ instead of $T(h)$, we obtain
$$\sup_{k \in K}|(l_{k^{-1}}l_{Y_\alpha}T)(h)| \leq C'  \sum_{|\beta|+|\gamma| \leq m+|\alpha|} \sup_{g\in \overline{B}_r(o)} |(l_{X_\beta} (l_{H_\gamma} h))(g)|, \;\;\;\; \forall h \in C^\infty(X,\E_{\tilde{\tau}}).$$
In fact, since $K$ is compact and operates continuously on $C^{-\infty}_c(X,\E_\tau)$, the constant $C'>0$ can be chosen to be independently of $K$.
Moreover, $h$ is annihilated by each $l_{X_\beta}$ for $\beta \neq 0$ and it is an eigenfunction of each $l_{H_\gamma}$ with eigenvalue a polynomial in $\lambda \in \aL^*_\C$ of degree $\leq |\gamma|$, i.e.
$$|l_{Y_\alpha} \mathcal{F}_\tau(T)(\lambda,k)|=|(l_{k}l_{Y_\alpha}T)(e^\tau_{\lambda,1})|
\leq C_{r,N,\alpha} (1+|\lambda|^2)^{N+\frac{|\alpha|}{2}} e^{r|\Repart(\lambda)|},$$
for $N \geq \frac{m}{2}.$
This complete the proof. \qedhere
\end{itemize}
\end{proof}

Consequently, by (\ref{eq:Th}), the \textit{inverse Fourier transform} of $\psi \in PWS_{\tau}(\aL^*_\C \times K/M)$ for a test function $h \in C^\infty_c(X,\E_{\tilde{\tau}})$ is given by
$$\langle \mathcal{F}_\tau^{-1}(\psi),h\rangle:=\sum_{Q \in \mathcal{Q}} \sum_{\nu \in A^\tau_Q} \int_{i\aL^*_Q} \int_K \langle \mathcal{F}_{\tilde{\tau}} (h)(-\nu-\lambda,k),\;
\mu^Q_\nu(\lambda) \psi(\nu+\lambda,k)
 \rangle \; dk \; d\lambda.$$

Finally, we discuss the topology on the image space by which the Fourier transform becomes a topological isomorphism.

\begin{lem}  \label{lem:topPWS}
\begin{itemize}
\item[(a)] The Fourier transform
$\mathcal{F}_\tau: C^{-\infty}_c(X,\E_\tau) \longrightarrow PWS_\tau(\aL^*_\C \times K/M)$
is continuous.
\item[(b)] The inverse Fourier transform
\begin{equation} \label{eq:FTinverseThm}
\mathcal{F}_\tau^{-1}: PWS_{\tau}(\aL^*_\C \times K/M) \longrightarrow C_c^{-\infty}(X,\E_\tau)
\end{equation}
 is continuous.
\end{itemize}
\end{lem}

\begin{proof}
\begin{itemize}
\item[(a)]  
We will show that for each bounded $B'\subset C^{-\infty}_c(X,\E_\tau)$, there exist $r\geq 0$ and $N\in \N_0$ such that $\mathcal{F}_\tau(B')$ is contained as a bounded set in $PWS_{\tau,r,N}$.
Since $PWS_{\tau,r,N} \hookrightarrow PWS_{\tau}(\aL^*_\C \times K/M)$ is continuous, by definition of inductive limit, then $\mathcal{F}_\tau(B')$ is also bounded in  $PWS_{\tau}(\aL^*_\C \times K/M)$.
By Lem.~\ref{lem:Montelspace}, we will have that $\mathcal{F}_\tau$ is continuous.

Now let $B'\subset C^{-\infty}_c(X,\E_\tau)$ be bounded. 
Since $B'$ is equicontinuous and because of (\ref{eq:seminormp}), there exist $r \geq 0, m \in \N_0$ and a constant $C>0$ such that (\ref{eq:HelgasonIneq}) holds uniformly for all $T \in B'$:
$$\sup_{T \in B'} p_B(T)= \sup_{T \in B',\; \varphi \in B}|T(\varphi)| \leq C \sum_{|\alpha| \leq m} \sup_{g \in \overline{B}_r(o)} |l_{Y_\alpha} \varphi(g)|.$$
Now by arguing as in the proof of Prop.~\ref{lem:isoPWS} (c), we obtain, for $N=[\frac{m}{2}]$ that
$$||\mathcal{F}_\tau(T)||_{r,N,\alpha} \leq \infty, \;\;\;\forall T \in B'$$
i.e., $\mathcal{F}_\tau(B') \subset PWS_{\tau,r,N}$ is bounded.
Hence the Fourier transform is continuous.

\item[(b)] 
It suffices to show that if, for all $r \geq 0$ and $N \in \N_0$
\begin{equation} \label{eq:InverseFT1}
\mathcal{F}_\tau^{-1}: PWS_{\tau,r,N}(\aL^*_\C \times K/M) \longrightarrow C^{-\infty}(X,\E_\tau)
\end{equation}
is continuous. Indeed, by construction of the inductive limit topology and the remark between (\ref{eq:seminormTB}) \& (\ref{eq:boundedT}), as well as using $\mathcal{F}_\tau^{-1}(PWS_{\tau,r,N}) \subset C^{-\infty}_r(X,\E_\tau)$, we have that (\ref{eq:FTinverseThm})
is continuous.\\
Fix $r\geq 0$ and $N \in \N_0$. We want to show that (\ref{eq:InverseFT1}) is continuous.
For that, it suffices to show that for every bounded $\tilde{B} \subset C^{\infty}_c(X,\E_{\tilde{\tau}})$, we have
$$p_{\tilde{B}}(\mathcal{F}^{-1}_\tau(\psi)) \leq C||\psi||_{r,N,0} (< \infty), \;\;\;\; \psi\in PWS_{\tau,r,N},$$ 
where $p_{\tilde{B}}(\cdot)$ is the seminorm as in (\ref{eq:seminormTB}) and $C$ is a positive constant.
Since $\tilde{B}$ is bounded subset in $C^{\infty}_c(X,\E_{\tilde{\tau}})$, there exsits $R \geq 0$ so that the support of all $\varphi \in \tilde{B}$ are in $\overline{B}_R(o)$.
Thus, for $\psi\in PWS_{\tau,r,N}$, we have that
\begin{eqnarray*}
&&p_{\tilde{B}}( \mathcal{F}^{-1}_\tau(\psi)) \\
&\stackrel{(\ref{eq:seminormTB})}{=}& \sup_{\varphi \in \tilde{B}}| \langle \mathcal{F}^{-1}_\tau(\psi), \varphi \rangle | \\
&\stackrel{(\ref{eq:HCPformula})}{=}&  \sup_{\varphi \in \tilde{B}} \Big| \sum_{Q \in \mathcal{Q}} \sum_{\nu \in A^\tau_Q} \int_{i\aL^*_Q} \int_K \langle 
\mathcal{F}_{\tilde{\tau}}(\varphi)(-\nu-\lambda,k) \;,\;
\mu^Q_\nu(\lambda) \psi(\nu+\lambda,k) \rangle \; dk d\lambda \Big| \\
&\leq&  \sup_{\varphi \in \tilde{B}} \sum_{Q \in \mathcal{Q}} \sum_{\nu \in A^\tau_Q} \int_{i\aL^*_Q} \int_K \Big| \langle 
\mathcal{F}_{\tilde{\tau}}(\varphi)(-\nu-\lambda,k) \;,\;
\mu^Q_\nu(\lambda) \psi(\nu+\lambda,k) \rangle\Big| \; dk d\lambda.
\end{eqnarray*}
Fix now $Q \in \mathcal{Q}$ and $\nu \in A^\tau_Q$.
Set $$d_{Q,\nu}:=\sup_{\varphi \in \tilde{B}} \int_{i\aL^*_Q} \int_K \Big| \langle 
\mathcal{F}_{\tilde{\tau}}(\varphi)(-\nu-\lambda,k) \;,\;
\mu^Q_\nu(\lambda) \psi(\nu+\lambda,k) \rangle\Big| \; dk d\lambda.$$
It suffices to show that $d_{Q,\nu} \leq C ||\psi||_{r,N,0}.$
We have
\begin{eqnarray*}
d_{Q,\nu} &\leq&  \sup_{\varphi \in \tilde{B}} \int_{i\aL^*_Q} \int_K (1+|{\nu+}\lambda|^2)^{-d_Q}(1+|{\nu+}\lambda|^2)^{d_Q} 
|\mathcal{F}_{\tilde{\tau}}(\varphi)(-\nu-\lambda,k)|\\
&&\;\;\;\;\;\;\;\;\;\;\;\;\;\;\;\;\;\;\;\;\;\;\;\;\;\;\;\;\;\;\;\;\;\;\;\;\;\;\;\;\;\;\;\;\;\;\;\;\;\;\;\;\;\;\;\;\;\;\;\;\;\;\; |\mu^Q_\nu(\lambda) \psi(\nu+\lambda,k) | \; dk d\lambda\\
&\leq& C \sup_{\substack{\varphi \in \tilde{B}\\ k \in K, \lambda \in i\aL^*_Q}}
 (1+|{\nu+}\lambda|^2)^{d_Q} 
|\mathcal{F}_{\tilde{\tau}}(\varphi)(-\nu-\lambda,k)| \;
|\mu^Q_\nu(\lambda) \psi(\nu+\lambda,k)|
\end{eqnarray*}
where $C:=\int_{i\aL^*_Q} (1+|{\nu+}\lambda|^2)^{-d_Q} \;d\lambda < \infty$ and 
$(1+|{\nu+}\lambda|^2)^{d_Q}$ is a weight factor with some $d_Q\in \N_0$ depending on the dimension of $i\aL^*_Q$.
For some positive constant $N$ and growth constant $m \in \N_0$, we get
\begin{eqnarray*}
d_{Q,\nu}
&\leq&C\sup_{\substack{\varphi \in \tilde{B}\\ k \in K, \lambda \in i\aL^*_Q}}
 (1+|{\nu+}\lambda|^2)^{d_Q+N+m} 
|\mathcal{F}_{\tilde{\tau}}(\varphi)(-\nu-\lambda,k)|  \\
&&\cdot \sup_{k \in K, \lambda \in i\aL^*_Q}
 (1+|{\nu+}\lambda|^2)^{-(N+m)} 
|\mu^Q_\nu(\lambda) \psi(\nu+\lambda,k)| \\
&\leq& C'\sup_{\substack{\varphi \in \tilde{B}\\ k \in K, \lambda \in i\aL^*_Q}}
 (1+|{\nu+}\lambda|^2)^{d_Q+N+m} 
|\mathcal{F}_{\tilde{\tau}}(\varphi)(-\nu-\lambda,k)|  \\
&&\cdot\sup_{k \in K, \lambda \in i\aL^*_\C} (1+|\nu+\lambda|^2)^{-N} | \psi(\nu+\lambda,k)|,
\end{eqnarray*}
where $||\mu^Q_\nu(\lambda)||_{\text{op}} \leq C'(1+|{\nu+}\lambda|^2)^m$ of at most polynomial growth of $m \in \N_0$.
Thus
\begin{eqnarray*}
d_{Q,\nu}&\leq&C'' \sup_{\substack{\varphi \in \tilde{B}\\ k \in K, \lambda \in i\aL^*_Q}} e^{R|\nu|}
(1+|{\nu+}\lambda|^2)^{d_Q+N+m}
|\mathcal{F}_{\tilde{\tau}}(\varphi)({-\nu}-\lambda,k)| \\
&&\cdot \sup_{k \in K,\lambda \in i\aL^*_Q} e^{r|\nu|}
(1+|{\nu+}\lambda|^2)^{-N} | \psi(\nu+\lambda,k)| \\
&=&C'' \sup_{\varphi \in \tilde{B}} ||\mathcal{F}_{\tilde{\tau}}(\varphi)||_{R,d_Q+N+m} ||\psi||_{r,N,0},
\end{eqnarray*}
where we set $\xi:=\nu+\lambda \in \aL^*_\C$. By the Paley-Wiener Thm.~\ref{thm:PWsect}, $\mathcal{F}_{\tilde{\tau}}$ is continuous, thus
 $$\sup_{\varphi \in \tilde{B}} ||\mathcal{F}_{\tilde{\tau}}(\varphi)||_{R,d_Q+N+m} < \widetilde{C} < \infty.$$
Therefore, $d_{Q,\nu} \leq C'''|\psi||_{r,N,0}$ and hence the inverse Fourier transform is continuous. \qedhere
\end{itemize}
\end{proof}

\begin{proof}[End of the proof of Thm.~\ref{thm:PWSsect}]
The isomorphism of the Fourier transform map outcomes from  Prop.~\ref{lem:isoPWS} and the continuity and topology statement results from Lem.~\ref{lem:topPWS}, hence this completes the proof.\\
Analogously, we obtain the topological Fourier isomorphism in (\textcolor{blue}{Level 3}) by taking 
$C^{-\infty}_c(G,\gamma,\tau)$ instead of $C^{-\infty}_c(X,\E_\tau)$.
\end{proof}

\section{Invariant differential operators on the Fourier range}
\label{sect:ImpactDO}
We consider the vector space of distributional sections  $C^{-\infty}_{\{o\}}(X,\E_\tau)$ supported at the origin $o=eK\in X$.
Since $g\cdot o \neq o$, $G$ does not act on $C^{-\infty}_{\{o\}}(X,\E_\tau)$, but $K$ as well as $\g$ do, thus $C^{-\infty}_{\{o\}}(X,\E_\tau)$ is a $(\g,K)$-module (e.g. \cite{Wallach}, 3.3.1). 
Moreover, it is generated by the so-called \textit{vector-valued Dirac delta-distributions} $\delta_v$ at $v\in E_\tau$:
$$\delta_v(f)=\langle v, f(e) \rangle_\tau, \;\;\; \text{with test function } f \in C^{\infty}_{(c)}(X, \E_{\tilde{\tau}}),$$
where $\langle \cdot, \cdot \rangle_\tau$ denotes the pairing in $E_\tau$.
In particular, we have the following identification:
$$\U(\g) \otimes_{\U(\kk)} E_\tau \stackrel{\beta}{\cong} C^{-\infty}_{\{o\}}(X,\E_\tau)$$
given by
$ \beta(Z \otimes v)(f):= \langle r_Zf(e),v \rangle_\tau,$ for $Z \in \U(\g), v \in E_\tau, f\in C^\infty(X,\E_{\tilde{\tau}}),$
with actions
$
Y(Z \otimes v) = YZ \otimes v,$ and
$k(Z \otimes v) = \Ad(k)Z \otimes \tau(k)v,$ for $Y \in \kk \;(\text{or } \U(\kk)), k\in K.$\\
In addition, every invariant differential operator $D \in \DD_G(\E_\gamma,\E_\tau)$ may be viewed as a linear map between these spaces 
$D: C^{-\infty}_{\{o\}}(X,\E_\gamma) \longrightarrow C^{-\infty}_{\{o\}}(X,\E_\tau)$.
This map defines an element
$$H_D \in \Hom_{K}(E_\gamma, C^{-\infty}_{\{o\}}(X,\E_\tau)) {\cong} [C^{-\infty}_{\{o\}}(X,\E_\tau) \otimes E_{\tilde{\gamma}}]^{K}$$
given by
\begin{equation} \label{eq:5}
H_D(v):= D( \delta_v) \in C^{-\infty}_{\{o\}}(X,\mathbb{E}_\tau), \;\;\;\; v \in E_\gamma, \delta_v \in C^{-\infty}_{\{o\}}(X,\E_\gamma).
\end{equation}
In other words
\begin{equation} \label{eq:6}
 \langle H_D(v), f \rangle_\tau \stackrel{(\ref{eq:5})}{=} \langle \delta_v, D^t (f )\rangle_\gamma = \langle v, D^t (f) (1) \rangle_\gamma,
\end{equation}
where $D^t \in \DD_G(\E_{\tilde{\tau}}, \E_{\tilde{\gamma}})$ is the adjoint invariant differential operator of $D$ defined by the corresponding pairing.
Since the graded space of both Hilbert spaces $ \DD_G(\E_{\tilde{\tau}}, \E_{\tilde{\gamma}})$ and $\Hom_{K}(E_\gamma, C^{-\infty}_{\{o\}}(X,\E_\tau))$ is isomorphic to $[S(\mathfrak{p}) \otimes \Hom(E_\gamma,E_\tau)]^{K}$, we have the following isomorphism:
\begin{eqnarray*}
\DD_G(\E_\gamma, \E_\tau) &\tilde{\longrightarrow}& \Hom_{K}(E_\gamma, C^{-\infty}_{\{o\}}(X,\E_\tau))\\
D &\mapsto& H_D. 
\end{eqnarray*}
Here, $S(\mathfrak{p})$ denotes the symmetric algebra of $\mathfrak{p} \subset \g$.\\
Consequently, we have 
$\DD_G(\E_\gamma,\E_\tau) \cong \Hom_K(E_\gamma, C^{-\infty}_{\{0\}}(X,\E_\tau)) 
\cong C^{-\infty}_{\{0\}}(G,\gamma,\tau).$
Hence, by applying the Fourier transform in (\textcolor{blue}{Level 3}) and the Paley-Wiener-Schwartz Thm.~\ref{thm:PWSsect} (b), we can deduce the following result.

\begin{prop} \label{prop:PolPWS}
With the notations above, we then have
\begin{eqnarray*}
\prescript{}{\gamma}{\mathcal{F}}_\tau(\DD_G(\E_\gamma,\E_\tau)) 
&\cong& \prescript{}{\gamma}{PWS}_{\tau,0}(\aL^*_\C)\\
&=&\{P \in \emph \Pol(\aL^*_\C, \emph \Hom_M(E_\gamma, E_\tau)) \;|\; P \text{ satisfies } (3.iii) \text{ of Def.~\ref{def:PWSspace}}\}. \QEDA
\end{eqnarray*} 
\end{prop}
Thus, provided one has a good understanding of the intertwining condition $(3.iii)$ , one can determine $\DD_G(\E_\gamma,\E_\tau)$.
The converse holds by van den Ban's and Souaifi's Lem.~5.3 and Cor.~5.4 in \cite{vandenBan}.
Strictly speaking these results are in terms of the Hecke algebra (\ref{eq:Heckealg}).
But the $(\gamma,\tilde{\tau})$-isotypic component $\mathcal{H}(G,K)(\gamma \otimes \tilde{\tau})$ of the Hecke algebra  is exactly $\DD_G(\E_\gamma,\E_\tau) \otimes \Hom(E_\tau,E_\gamma)$.
In other words, given all invariant differential operators $D\in \DD_G(\E_\gamma, \E_\tau)$, one can determine explicitly the intertwining condition $(3.iii)$ and the corresponding Paley-Wiener space.

Moreover, we remark that the isomorphism in Prop.~\ref{prop:PolPWS} can also be described more algebraically as a Harish-Chandra type homomorphism, we refer to (\cite{Martin}, p.~4) or (\cite{Palmirotta}, Sect.~2.1) for more details.

In addition, we also have the following result.

\begin{prop} \label{prop:1}
Let $D \in \DD_G(\E_\gamma, \E_\tau)$ be an invariant linear differential operator.
For $f \in C^{\pm \infty}_c(X,\E_{\gamma})$, we then have that
\begin{equation} \label{eq:Thm1}
\mathcal{F}_\tau (Df)(\lambda,k)= \prescript{}{\gamma}{\mathcal{F}}_\tau(H_D)(\lambda) \mathcal{F}_{{\gamma}} (f)(\lambda,k), \;\;\;\;\;  \lambda \in \aL^*_\C, k\in K,
\end{equation}
where $\prescript{}{\gamma}{\mathcal{F}}_\tau(H_D) \in \emph \Pol(\aL^*_\C, \emph \Hom_M(E_\gamma, E_\tau))$ is a polynomial in $\lambda \in \aL^*_\C$ with values in $\emph \Hom_M(E_\gamma, E_\tau)$.
\end{prop}

\begin{proof}
We know that the Fourier transform of a distribution $H_D \in \Hom_K(E_\gamma, C^{-\infty}_{\{o\}}(X,\E_\tau))$ is defined by
$\prescript{}{\gamma}{\mathcal{F}}_\tau (H_D)(\lambda)(v)= \langle H_D(v), e^\tau_{\lambda,1} \rangle,$
for $v\in E_\gamma$ and where $e^\tau_{\lambda,1} \in C^\infty(G,\tau,\tilde{\tau})$. Hence by (\ref{eq:6}), we obtain
\begin{equation} \label{eq:step1}
\prescript{}{\gamma}{\mathcal{F}}_\tau (H_D)(\lambda)(v)
= \langle H_D(v), e^\tau_{\lambda,1} \rangle _\tau
\stackrel{(\ref{eq:6})}{=} \langle v, D^t (e^\tau_{\lambda,1})(1) \rangle_\gamma
= (D^t (e^\tau_{\lambda,1})(1))v, \;\;\; \lambda \in \aL^*_\C.
\end{equation}
Now,  by considering a function $f\in C^\infty_c(X,\E_\gamma)$,
we  conclude, via 'partial integration', that (\ref{eq:Thm1}) holds.
In fact
\begin{eqnarray*}
\mathcal{F}_\tau(Df)(\lambda,k) = \int_{G} e^\tau_{\lambda,k}(g) D(f(g)) \;dg
&\stackrel{\text{def. of }D^t}{=}& \int_{G} D^t (e^\tau_{\lambda,k}(g)) f(g) \;dg \\
&\stackrel{(\ref{eq:step2})}{=}& \int_{G} D^t (e^\tau_{\lambda,1}(1)) \circ e^\gamma_{\lambda,k}(g) f(g) \;dg \\
&=& D^t(e^\tau_{\lambda,1}(1)) \circ \mathcal{F}_\gamma (f)(\lambda,k) \\
&\stackrel{(\ref{eq:step1})}{=}& \prescript{}{\gamma}{\mathcal{F}}_\tau(H_D)(\lambda) \circ \mathcal{F}_\gamma (f)(\lambda,k).
\end{eqnarray*}
The same computation remains true for $f\in C^{-\infty}_c(X,\E_\gamma)$, by using
 the pairing $\langle \cdot, \cdot \rangle$ instead of the integration. 
\end{proof}

\begin{req}
Consider an additional not necessarily irreducible $K$-representation $(\delta, E_\delta)$. Then, 
for $D_1 \in \DD_G(\E_\tau,\E_\delta)$ and $D_2 \in \DD_G(\E_\gamma,\E_\tau)$,
Prop.~\ref{prop:1} implies that
$$
\prescript{}{\gamma}{\mathcal{F}}_\delta (H_{D_1} \circ H_{D_2}) = \prescript{}{\tau}{\mathcal{F}}_\delta(H_{D_1}) \circ \prescript{}{\gamma}{\mathcal{F}}_\tau(H_{D_2}).$$
\end{req}

\subsection*{Acknowledgement}
%We are grateful to Salah Mehdi for his helpful feedbacks.
This work is supported by the Fond National de la Recherche, Luxembourg under the project code: PRIDE15/10949314/GSM.

\medskip
 %A

\begin{minipage}[t][2.5cm][b]{0.7\textwidth}
Université du Luxembourg,\\
Faculty of Science, Technology and Medicine,\\
Department of Mathematics \\
Email addresses: \href{mailto:guendalina.palmirotta@uni.lu}{guendalina.palmirotta@uni.lu} \& \href{mailto:martin.olbrich@uni.lu}{martin.olbrich@uni.lu}
\end{minipage}

\end{document}